\documentclass[11pt]{article}
\usepackage{amsmath,amsfonts,amssymb, amsthm,enumerate,graphicx}
\usepackage{wrapfig}
\usepackage{lipsum}
\usepackage{tikz,authblk}
\usepackage{float}
\usepackage{hyperref}
\usepackage{subfig}
\usepackage{url}
\usetikzlibrary{shapes.geometric}
\newtheorem{theorem} {{\textsf{Theorem}}}
\newtheorem{proposition}[theorem]{{\textsf{Proposition}}}
\newtheorem{corollary}[theorem]{{\textsf{Corollary}}}
\newtheorem{conj}[theorem]{{\textsf{Conjecture}}}
\newtheorem{definition}[theorem]{{\textsf{Definition}}}
\newtheorem{remark}[theorem]{{\textsf{Remark}}}
\newtheorem{example}[theorem]{{\textsf{Example}}}
\newtheorem{lemma}[theorem]{{\textsf{Lemma}}}
\newtheorem{question}[theorem]{{\textsf{Question}}}

\newcommand{\lk}{\mbox{\upshape lk}\,}
\newcommand{\Sp}{\mathbb{S}}

\begin{document}
\title{Simplicial degree $d$ self-maps on $n$-spheres}
\author{Biplab Basak$^1$, Raju Kumar Gupta$^*$, and Ayushi Trivedi}
\date{}
\maketitle
\vspace{-10mm}
\begin{center}

\noindent {\small Department of Mathematics, IIT Delhi, Hauz Khas, New Delhi 110016.}

\noindent {\small Mathematical Institute, Slovak Academy of Sciences, Bratislava, 81438$^{*}$.}

\noindent {\small {\em E-mail addresses:} \url{biplab@iitd.ac.in}, \url{rajukrg3217@gmail.com}$^{*}$ \url{Ayushi.Trivedi@maths.iitd.ac.in}}

\medskip

\date{\today}
\end{center}
\footnotetext[1]{Corresponding author}
\hrule

\begin{abstract}
The degree of a map between orientable manifolds is a fundamental concept in topology, providing deep insights into the structure of manifolds and the behavior of maps between them. Recently, this notion has been extensively studied, particularly in the context of simplicial maps between orientable triangulable spaces. 

In this paper, we focus on the construction of non-degenerate simplicial maps of degree $d\in \mathbb{Z}$ on $n$-spheres for $n\geq 2$. We develop a general method, based on connected sums and facet orientations, to construct simplicial maps of any prescribed degree $d \in \mathbb{Z}$ between triangulated spheres. 
We investigate the asymptotic behavior of $\Lambda(n,d)$, defined as the minimum number of vertices required for a triangulated $n$-sphere to admit a simplicial map of degree $d$ to $\Sp^n_{n+2}$, for $n \geq 3$ and $d \geq 1$. As a consequence, we answer a question posed by Ryabichev in \cite{Ryabichev}. In addition to vertex-minimal constructions, we obtain facet-minimal degree maps for large degrees. Specifically, for each $d \geq n^2 + 1$, we construct a simplicial map of degree $d$ from a triangulated $n$-sphere with $d(n+2)$ facets to $\Sp^n_{n+2}$, for $n \geq 3$. 

As an application of the constructions, we derive improved bounds on the covering type of Moore spaces, refining results from \cite{Govc}. Finally, we conclude with several open questions that may be of independent interest.

\end{abstract}

\noindent {\small {\em MSC 2020\,:}  Primary 05E45; Secondary 57Q15, 55M25, 52B70.
	
}

\section{Introduction}
The interplay between topology and combinatorics has been a rich area of mathematical study, providing insights that benefit both disciplines. This connection is particularly evident in the concept of triangulations of manifolds, which serve as a bridge between these two fields. Triangulations are useful in both theoretical and computational mathematics.
It involves the decomposition of a manifold into simpler geometric shapes, known as simplices, which allows for a combinatorial approach to studying the shape and structure of topological spaces. The construction of new triangulations and the exploration of combinatorial aspects of topological spaces and their corresponding maps have captivated researchers for several decades. 

Let $X$ and $Y$ be two closed orientable $n$-manifolds.
It is known that $H_n(X, \mathbb{Z}) \cong  H_n(Y, \mathbb{Z}) \cong  \mathbb{Z}$. Let $f:X\to Y$ be a continuous map. The induced homomorphism $f_{*}: H_n(X, \mathbb{Z}) \rightarrow H_n(Y, \mathbb{Z})$ can be expressed as $f_{*}([X]) = d \cdot [Y]$, where $[X]$ and $[Y]$ are the generators of the $n$th homology groups of $X$ and $Y$, respectively. The integer $d$ is called the degree of the map $f$.
The degree of a map $f: X \rightarrow Y$ is one of the fundamental invariants for studying continuous maps between closed orientable manifolds. Degree of a map bridges various areas of topology, geometry, and mathematical physics. In mathematical physics, the degree can signify quantities like the winding number, topological charge, or the flux of a field through a manifold. The study of such degree maps has a long history in topology, originating from Gauss’s ideas in his proof of the fundamental theorem of algebra.  The theory of degree maps was developed in the works of Hopf \cite{Hopf1928, Hopf1930} and further advanced by Olum \cite{Olum1953} and Epstein \cite{Epstein1966}. 
Recently, in \cite{Ryabichev2024}, the author proved that for given closed, possibly nonorientable surfaces $M$ and $N$, if a map $f : M \to N$ has geometric degree $d > 0$, then $\chi(M) \leq d \cdot \chi(N)$. This categorizes the potential surfaces that can have a degree map defined between them. Some recent study on degree maps between manifolds can be found in \cite{Amann, Neofytidis, Neofytidis2}. 

Recently, the study of maps between manifolds in a combinatorial setup has seen some interest  (see \cite{Apolonskaya,BasakTrivedi,Milizia,Musin,Ryabichev}).
The combinatorial version of the study on degree maps was initiated by Fan in \cite{Fan1967},  where he studied the simplicial maps from an orientable $n$-pseudomanifold to an $m$-sphere with the octahedral triangulation. Let $\Sp^n_{n+2}$ denote the standard $n$-sphere with $n+2$ vertices. In \cite{Madahar2001},  for a given $d\in \mathbb{Z}$,  authors, constructed triangulation of a $2$-sphere $K$ on $2|d|+2$ vertices such that there exists a simplicial map $f: K \rightarrow \Sp^2_{4}$ of degree $d$. They also proved that the triangulations obtained are vertex minimal.

Let $K$ be an oriented triangulated $n$-sphere, and let $f:K \to \Sp^n_{n+2}$ be a simplicial map. We call the map $f$ \emph{non-degenerate} if it is surjective and maps $n$-simplices of $K$ onto $n$-simplices of $\Sp^n_{n+2}$; otherwise, we call it \emph{degenerate}. If a simplicial degree map $f:K \to \Sp^n_{n+2}$ is non-degenerate, then the $1$-skeleton of $K$ is properly $(n+2)$-colorable, which need not hold in the degenerate case.
In this article, we study simplicial degree maps that are non-degenerate.

In Section~\ref{preliminaries}, we present the basic definitions and known results required for the article. In Section~\ref{section2}, we begin by describing a construction method using connected sums and orientations on simplices of triangulated spheres induced from the orientation on the standard sphere $\Sp^n_{n+2}$ and the simplicial map. This method allows us to construct simplicial degree $d = d' + d''$ maps from given simplicial degree maps of degrees $d'$ and $d''$, for $n \geq 2$ and $d \geq 1$. We first apply this method to construct triangulated $2$-spheres that induce simplicial degree $d$ maps for all $d\geq 1$ (see  Proposition \ref{2sphere}). The resulting $2$-spheres differ combinatorially from those in \cite{Madahar2001}.
Next, we construct simplicial degree $d$ maps $f:K \to \Sp^n_{n+2}$ for every $d\in \mathbb{Z}$, where $K$ is a triangulated $n$-sphere (see Theorem~\ref{thm:mainconstruction}).

Fix $n,d \geq 1$. Let $\lambda(n,d)$ (respectively, $\Lambda(n,d)$) denote the minimum number of vertices required for a triangulated $n$-sphere to admit a simplicial degree map (respectively, a non-degenerate simplicial degree map) $f:K \to \Sp^n_{n+2}$ of degree $d$. Clearly, $\lambda(n,d) \leq \Lambda(n,d)$ for all $n,d \geq 1$. It is known that $\lambda(1,d) = \Lambda(1,d) = 3d$, and from \cite{Madahar2001}, $\lambda(2,d) = \Lambda(2,d) = 2d + 2$ for $d \geq 1$.

In \cite{Ryabichev} (see Proposition~1), the author used the join operation to construct triangulated spheres with fewer vertices that induce maps of large degree. However, due to the use of the join operation, these maps are degenerate. The author obtained the bound $\lambda(n,d_1d_2) \leq 3d_1 + 3d_2 + n - 2$ and also proved that $\limsup_{d \to \infty} \frac{\lambda(n,d)}{d} = 0$. This suggests that determining the exact values of $\lambda(n,d)$ for given $n$ and $d$ is extremely challenging.

In this article, we provide a lower bound for $\Lambda(n,d)$ for $n \geq 3$ and study its asymptotic behavior with respect to $n$ and $d$ (see Subsection~\ref{subsection 1}).
In \cite{Ryabichev}, the author also proved that
$\limsup_{d \to \infty} \frac{(\lambda(n,d))^{l}}{d} = 0$
for $l < \lfloor \frac{n+1}{2} \rfloor$, and posed the following question:

\begin{question}
Is 
$\limsup_{d \to \infty}
\frac{(\lambda(n,d))^{\lfloor \frac{n+1}{2} \rfloor}}{d} \neq 0$
for $n \geq 3$?
\end{question}

In this article, we answer this question in the affirmative (see Theorem~\ref{openquestion}).

If $f:K \to \Sp^n_{n+2}$ is a simplicial map of degree $d\in \mathbb{Z}$, where $n \geq 1$, then $K$ must have at least $d(n+2)$ facets. It is clear that  $f$ is a \emph{facet-minimal simplicial degree map} if $K$ has exactly $d(n+2)$ facets. 

Apart from exploring vertex minimality, we obtain several interesting results on facet-minimal simplicial degree maps. We prove that there exists a simplicial degree map 
$g:L \to \Sp^n_{n+2}$ 
of degree $d$ for $d \geq n^2+1$, $d = ln+1$ (for $l \geq 1$), and $d = n(n+1-r)+r$ (for $1 \le r \le n$), such that $f_n(L) = d(n+2)$ (see Theorems~\ref{thm:mainconstruction} and~\ref{facetminimal}). We also show that such maps are not realizable for degrees $d = 2$ and $d = 3$.

In Section~\ref{section3}, we characterize simplicial degree maps and use this characterization to strengthen a result of Planken \cite{Planken}. Planken proved that if the $1$-skeleton of a triangulated $n$-sphere is $(n+2)$-colorable, then there exists a subdivision of the sphere in which, for every $(n-2)$-simplex, the number of incident $(n-1)$-simplices is divisible by three.
In particular, we show that there exists a subdivision of such triangulated spheres in which, for every $(n-2)$-simplex, the number of incident $(n-1)$-simplices is divisible by three, and for every $(n-3)$-simplex, the number of incident $(n-2)$-simplices is even and cannot be equal to $6$.

In Section~\ref{section4}, we present an application of our constructions of simplicial degree maps. 
Karoubi and Weibel \cite{Karoubi} introduced the notion of the \emph{strict covering type} of a space $X$, defined as the minimal cardinality of a good cover of $X$. Since the strict covering type is a geometric notion and not homotopy invariant, they further defined the \emph{covering type} of $X$ by $ct(X) := \min\{\mathrm{sct}(Y) \mid Y \simeq X\}$.
They also proved (see \cite[Theorem 2.5]{Karoubi}) that the covering type of a finite CW complex is equal to the minimal cardinality of a good closed cover of some CW complex homotopy equivalent to $X$. In \cite{Govc}, the authors obtained upper bounds on the covering type of Moore spaces (see Proposition~4.6 and Corollary~4.7 in \cite{Govc}). In this article, we significantly improve these bounds.
We use the fact that if $f:K \to \Sp^n_{n+2}$ is a simplicial degree map of degree $d$, then its mapping cone is a Moore space of type $M(\mathbb{Z}_d,n)$. Using the technique of simplicial mapping cones over $f$, as described in \cite{Tinarrage}, we construct a simplicial complex with $f_0(K) + n + 3$ vertices that is homotopy equivalent to the mapping cone of $f$. The results are presented in Propositions~\ref{pr:triangulation of moore space} and~\ref{pr2:Moore Spaces}.

Finally, in Section \ref{section5}, we conclude by addressing some interestiong open questions.

\section{Basic Terminologies}\label{preliminaries}
   A {\em $d$-simplex} is a $d$-dimensional polytope with exactly $d+1$ vertices. A simplex generated by the vertices $v_1, v_2, \ldots, v_{n},$ and $v_{n+1}$ is denoted by $[v_1 v_2 \cdots v_{n+1}]$. A {\em face} $\tau$ of a simplex $\sigma$ is the convex hull of a non-empty subset of the vertex set of $\sigma$ and is denoted by $\tau \leq \sigma$. A {\em simplicial complex} $K$ is a finite collection of simplices in $\mathbb{R}^m$ for some $m \in \mathbb{N}$ such that for any simplex $\sigma$, all of its faces are in $K$, and for any two simplices $\sigma$ and $\tau$ in $K$, either $\sigma \cap \tau$ is empty or a face of both simplices. We assume that the empty set is the face of every simplicial complex. The {\em dimension} of a simplicial complex $K$ is defined as $\max\{\dim(\sigma) : \sigma \in K\}$. The geometric carrier of $K$ is a compact polyhedron $|K|$ and is defined as $|K|:= \bigcup_{\sigma \in K} \sigma$. An {\em $f$-vector} of a $d$-simplicial complex $K$ is a $(d+2)$-tuple $(f_{-1}, f_0, f_1, \ldots, f_d)$ such that $f_{-1} = 1$, and for all $0 \leq i \leq d$, $f_i$ denotes the number of $i$-dimensional simplices in $K$. A {\em subcomplex} $S$ is a simplicial complex $S \subset K$.

    Let $\alpha = [u_1 u_2 \cdots u_{n+1}]$ and $\beta = [v_1 v_2 \cdots v_{m+1}]$ be two simplices of dimensions $n$ and $m$, respectively. Then the {\em join} of $\alpha$ and $\beta$ is the simplex $[u_1 u_2 \cdots u_{n+1} v_1 v_2 \cdots v_{m+1}]$ of dimension $n + m + 1$ and is denoted by $\alpha \star \beta$ or $\alpha \beta$. We define the {\em join} of two simplicial complexes $K_1$ and $K_2$ as $\{\sigma \star \tau : \sigma \in K_1, \tau \in K_2\}$, and we denote this as $K_1 \star K_2$. The dimension of $K_1 \star K_2$ is $\dim(K_1) + \dim(K_2) + 1$. 
The {\em link} of a face $\sigma$ in $K$ is defined as $\{\gamma \in K: \gamma \cap \sigma = \emptyset \text{ and } \gamma \star \sigma \in K\}$ and is denoted by $\operatorname{lk}(\sigma, K)$.

A triangulation of a polyhedron \(X\) is a simplicial complex \(K\) such that \(X\) is PL-homeomorphic to the geometric realization \(|K|\) of \(K\). If $M$ is a PL $n$-manifold and $|K|$ is PL-homeomorphic to $M$ for an $n$-simplicial complex $K$, then we say that $K$ is a triangulation of $M$ or a triangulated $n$-manifold. Let $K$ be a triangulated $n$-manifold with a connected boundary. Then $\partial(K)$ denotes the boundary of $K$ and is a triangulated $(n-1)$-manifold. Furthermore, $\partial(K)$ is the collection of $(n-1)$-faces of $K$ that are contained in exactly one $n$-face. By $\mathbb{S}^n_{n+2}$, we mean the standard oriented triangulation of the $n$-sphere with exactly $(n+2)$ vertices.  Let $K_1$ and $K_2$ be two simplicial complexes of the same dimension $n$. The \emph{connected sum} of $K_1$ and $K_2$, denoted by $K_1 \# K_2$, is obtained by choosing $n$-simplices $\sigma_1 \in K_1$ and $\sigma_2 \in K_2$, removing their interiors, and identifying their boundaries. The resulting complex $K_1 \# K_2$ is again an $n$-dimensional simplicial complex. Moreover, if $K_1$ and $K_2$ are triangulations of the $n$-sphere, then both $K_1 \# K_2$ is a triangulation of the $n$-sphere. Let $K$ be a simplicial complex and let $v \in f_0(K)$. The \emph{one-vertex suspension} of $K$ with respect to $v$, denoted by $\Sigma_v K$, is constructed as follows. Remove the vertex $v$ together with all faces incident to it from $K$. Then introduce two new vertices $x$ and $y$, and add the edge $xy$. Define the link of the edge $xy$ in $\Sigma_v K$ to be the same as the link of $v$ in $K$. Finally, for every face $\sigma$ of $K$ that does not belong to the star of $v$, include the faces $x\sigma$ and $y\sigma$ in $\Sigma_v K$. A simplicial complex $X$ is called vertex minimal if $f_0(X) \leq f_0(Y)$ for all triangulations $Y$ of the topological space $|X|$. An $n$-dimensional simplicial complex $X$ is called facet minimal if $f_n(X) \leq f_n(Y)$ for all triangulations $Y$ of the topological space $|X|$.

 A \emph{cyclic polytope} $C(n,m)$ is the convex hull of $n$ distinct points lying on a rational normal curve in $\mathbb{R}^{m}$. Theodore Motzkin conjectured the Upper Bound Theorem, which relates the $f$-vector of cyclic polytopes of dimension $m$ to the $f$-vector of any triangulated $(m-1)$-sphere. This conjecture was later proved by McMullen. The theorem is stated as follows.
 \begin{proposition}[\cite{mullen}]\label{pr:cyclicpolytope}
    Let $\Delta$ be a triangulation of the $m$-dimensional sphere with $n$ vertices; then for every $0\leq k \leq m$, $f_k(\Delta)\leq f_k(C(n,m+1))$, where $C(n,m+1)$ is the boundary complex of the cyclic polytope of dimension $m+1$ with $n$ vertices.  
\end{proposition}

\begin{proposition}[\cite{Bagchi}]\label{pr:pseudomanifold}
Let $K$ be a triangulated $n$-manifold. Then $f_n(K)\geq nf_0(K)-(n+2)(n-1)$.
\end{proposition}

\subsection{Degree of a simplicial map}

 Let $M$ and $N$ be two closed orientable triangulated $n$-manifolds, and let $f: M \rightarrow N$ be a map. Both $n$-th homology groups $H_n(M, \mathbb{Z}) $ and $H_n(N, \mathbb{Z}) $ are isomorphic to $ \mathbb{Z} $. The simplicial map $f$ induces a homomorphism $f_*: H_n(M, \mathbb{Z}) \to H_n(N, \mathbb{Z}) $. Let $ f_*([a]) = d \cdot [b] $, where $ H_n(M, \mathbb{Z}) = \langle [a] \rangle $ and $ H_n(N, \mathbb{Z}) = \langle [b] \rangle $, and $ d \in \mathbb{Z} $. The integer $d$ is called the degree of the map $f$.

A simplicial map between two simplicial complexes $K$ and $L$ is a map $f: K \rightarrow L$ such that if $[v_1 v_2 \cdots v_{m+1}]$ is a simplex in $K$, then $[f(v_1) f(v_2) \cdots f(v_{m+1})]$ is a simplex in $L$. Let $K$ be an oriented triangulation of an orientable $n$-manifold, i.e., each $n$-simplex is assigned an orientation (ordering of its vertices) such that any two adjacent $n$-simplices induce opposite orientations on their common $(n-1)$-face. Choose an $n$-simplex $\sigma = [v_1 v_2 \cdots v_{n+1}]$ in $K$ and declare it to be positive with respect to the ordering $v_1 < v_2 < \cdots < v_{n+1}$. Any representation of the simplex obtained from the ordered tuple $(v_1, v_2, \ldots, v_{n+1})$ by an even permutation represents the same orientation and is also called positive, whereas a representation obtained by an odd permutation represents the negative orientation. Suppose $\sigma_1 = [v_1 v_2 \cdots v_{i-1} v_i v_{i+1} \cdots v_{n+1}]$ and $\sigma_2 = [v_1 v_2 \cdots v_{i-1} v'_i v_{i+1} \cdots v_{n+1}]$ are two $n$-simplices sharing the common $(n-1)$-simplex $[v_1 v_2 \cdots v_{i-1} v_{i+1} \cdots v_{n+1}]$. Then $\sigma_1$ and $\sigma_2$ have opposite signs; that is, if one is positive, the other is negative with respect to the orderings $v_1 < v_2 < \cdots < v_i < \cdots < v_{n+1}$ and $v_1 < v_2 < \cdots < v'_i < \cdots < v_{n+1}$, respectively. Since \(K\) is oriented, this assignment is consistent; consequently, each \(n\)-simplex of \(K\) is uniquely classified as either positive or negative with respect to a chosen vertex ordering (up to even permutations).

\smallskip

 Let $K$ be an oriented triangulation of the $n$-sphere and let $f: K \rightarrow \mathbb{S}^n_{n+2}$ be a simplicial map. If $f$ is not surjective, then the degree of $f$ is zero. Moreover, if $f$ maps an $n$-simplex $\sigma \in K$ to an $i$-simplex with $i<n$, then $f(\sigma)$ does not contribute to the $n$-th homology group of $\mathbb{S}^n_{n+2}$. Therefore, we restrict our attention to \emph{non-degenerate simplicial maps}, that is, maps for which every $n$-simplex of $K$ is mapped onto an $n$-simplex of $\mathbb{S}^n_{n+2}$.
 For each simplex \(\sigma\) of \(\mathbb{S}^n_{n+2}\), let \(\alpha^{+}(\sigma)\) and \(\alpha^{-}(\sigma)\) denote the number of \(n\)-simplices of \(K\) mapped onto \(\sigma\) with the same sign and the opposite sign, respectively, with respect to the vertex ordering induced by the simplicial map. Then the degree of $f$ is 
   $\deg(f) = \alpha^{+}(\sigma) - \alpha^{-}(\sigma)$. We observe that the $\deg(f)$ is independent of the choice of the $n$-simplex $\sigma$ of $\mathbb{S}^n_{n+2}$. In this article, we are considering oriented triangulations of $n$-spheres.

\begin{remark}
 Let $K$ be an oriented triangulation of the $n$-sphere and let  $f: K \rightarrow \mathbb{S}^n_{n+2}$ be a simplicial map. Then the non-degeneracy of map $f$ implies that the one skeleton of $K$ is proper $(n+2)$-colorable.
\end{remark}

\begin{proposition}\label{Composition}
  Let $K, L$, and $M$ be triangulated closed orientable $n$-manifolds. If $f: K \rightarrow L$ and $g: L \rightarrow M$ are simplicial maps with degrees $d$ and $d'$, respectively, then the degree of the composite map $g \circ f$ is $d'\cdot d$.
\end{proposition}

\begin{remark}\label{remark}
{\rm Let $\{v_1, v_2, \dots, v_{n+2}\}$ be the set of vertices of $\Sp^n_{n+2}$. If $I: \Sp^n_{n+2} \rightarrow \Sp^n_{n+2}$ is the identity map, then clearly the degree of the map $I$ is $1$. Now, define a map $g: \Sp^n_{n+2} \rightarrow \Sp^n_{n+2}$ such that $g(v_i) = v_i$ for $1 \leq i \leq n$, $g(v_{n+1}) = v_{n+2}$, and $g(v_{n+2}) = v_{n+1}$. We observe that the degree of the map $g$ is $-1$. Now, if $f: K \rightarrow \Sp^n_{n+2}$ is a simplicial map of degree $d \geq 0$ from a triangulated $n$-sphere to the standard sphere $\Sp^n_{n+2}$, then from Proposition \ref{Composition}, the degree of the map $g \circ f$ is $-d$. Therefore, to construct a triangulated $n$-sphere that induces a simplicial degree $d$ map to $\Sp^n_{n+2}$, it is sufficient to consider cases where $d \geq 2$. From now on, we are only interested in constructing triangulations for maps of degrees $\geq 2$.}
\end{remark}

\begin{definition}
 Fix $n,d\geq 1$.  We define $\Lambda(n,d)$ as the minimal number of vertices required for a triangulated $n$-sphere $K$ such that there is a non-degenerate simplicial degree map $f:K\to \Sp^n_{n+2}$ of degree $d$.
\end{definition}

\begin{definition}
    Fix $n,d\geq 1$.  We define $\mathcal{F}(n,d)$ as the minimal number of facets required for a triangulated $n$-sphere $K$ such that there is a non-degenerate simplicial degree map $f:K\to \Sp^n_{n+2}$ of degree $d$.
\end{definition}

\subsection{Covering type of Moore spaces}
The concept of covering type of a space was introduced by Karoubi and Weibel in 2016 (see \cite{Karoubi}). The covering type of a space $X$ measures the topological complexity of $X$; it is defined as the cardinality of a good open cover of $X$. Simplicial maps of degree $d$ over the $n$-sphere play an important role in computing the covering type of the Moore space $M(\mathbb{Z}_d,n)$. Below, we present the necessary preliminaries required for calculating the covering type of the Moore space $M(\mathbb{Z}_d,n)$.

\begin{definition}[Moore Space] For a given abelian group $G$ and $n\geq 1$, a Moore space $M(G,n)$  is a connected CW complex such that $\widetilde{H}_n(M(G,n);\mathbb{Z})\cong G$ and $\widetilde{H}_i(M(G,n);\mathbb{Z})=0$ for all $i\neq n$.
\end{definition}

  \begin{definition}[Mapping Cone]
\rm{Let $f \colon K \to M$ be a continuous map between topological spaces $K$ and $M$. Denote by $C(K)$ the cone on the topological space $K$. 
The mapping cone of $f$ is defined by
$C(f) := M \cup_{f} C(K)$. The map $f$ is called the \emph{attaching map}.}
\end{definition}

 \begin{remark} 
\rm{ The mapping cone of a degree $d$ map $f:\mathbb{S}^n \to \mathbb{S}^n$ is the Moore space $M(\mathbb{Z}_d,n)$. 
Indeed, consider the CW complex obtained by attaching an $(n+1)$-cell to $\mathbb{S}^n$ via the map $f$. 
The cellular chain groups are $C_k(X)=\mathbb{Z}$ for $k=0,n,n+1$ and $C_k(X)=0$ otherwise. 
Since there are no $(n-1)$-cells, the boundary map $\partial_n$ is trivial, while the boundary map 
$\partial_{n+1}:C_{n+1}\to C_n$ is multiplication by $d$, induced by the attaching map $f$. 
Thus, the cellular chain complex is
\[
0 \longrightarrow \mathbb{Z} \xrightarrow{\times d} \mathbb{Z} \longrightarrow 0 .
\]
Hence $\widetilde H_n(X;\mathbb{Z})\cong \mathbb{Z}_d$ and all other reduced homology groups vanish. 
Therefore, the mapping cone of $f$ is the Moore space $M(\mathbb{Z}_d,n)$. }
\end{remark}

\begin{definition}
   \rm{Let \(X\) be a topological space.
The \emph{covering type} of \(X\), denoted by \(\mathrm{ct}(X)\), is the minimum cardinality of a good open cover of \(Y\).
Here, a \emph{good open cover} is a collection of open sets such that each set is contractible and every non-empty finite intersection of these sets is also contractible.
}
\end{definition}


We use the following lemma to compute the covering type of the Moore space $M(\mathbb{Z}_d,n)$.
\begin{lemma}\cite{Borghini}
    Let $X$ be a topological space homotopy equivalent to a finite CW-complex. Then, $\mathrm{ct}(X)$ coincides with the minimum possible number of vertices of a simplicial complex $K$ homotopy equivalent to $X$.
\end{lemma}

\section{Construction of degree maps over $n$-spheres}\label{section2}
In this section, we present constructions of triangulated $n$-spheres that induce a degree $d$ map for every $d \geq 2$. We begin by proving a lemma \ref{connectedlemma}. In this lemma, we analyze how the degree $d$ of the resulting construction is determined after performing the connected sum of the triangulations $K_1$ and $K_2$ of $n$-spheres, which induce maps of degrees $d_1$ and $d_2$, respectively.  

\begin{lemma}{\label{connectedlemma}}
Suppose we have two simplicial maps
$f_1 : K_1 \rightarrow \mathbb{S}^{n}_{n+2}
\quad\text{and}\quad
f_2 : K_2 \rightarrow \mathbb{S}^{n}_{n+2},$
with degrees \(d_1\) and \(d_2\), respectively for $d_1,d_2\geq 0$ and $n\geq 1$. Assume that for any one \(n\)-simplex \(\sigma\)  of among $n+2$ \(n\)-simplices \(\sigma\) in \(\mathbb{S}^{n}_{n+2}\), the preimage of \(\sigma\)
under \(f_1\) contains more than \(d_1\) \(n\)-simplices, or the preimage of
\(\sigma\) under \(f_2\) contains more than \(d_2\) \(n\)-simplices. Under these
conditions, we can define a new simplicial map
$f : K_1 \# K_2 \rightarrow \mathbb{S}^{n}_{n+2}$
whose degree is \(d_1 + d_2\).
\end{lemma}
\begin{proof}
As stated in the lemma, we have two simplicial maps $f_1$ and $f_2$, and we
assume that both maps use the same orientation on $\mathbb{S}^n_{n+2}$; in
particular, the signs of $n$-simplices in $\mathbb{S}^n_{n+2}$
agree for both $f_1$ and $f_2$. Let $\sigma \in \mathbb{S}^n_{n+2}$ be the
$n$-simplex considered in the lemma. Without loss of generality, assume that
$f_1^{-1}(\sigma)$ contains more than $d_1$ $n$-simplices.

Now, choose an $n$-simplex
$\sigma_1 \in f_1^{-1}(\sigma)$ in $K_1$ whose sign is opposite to that of
$\sigma$, and an $n$-simplex $\sigma_2 \in f_2^{-1}(\sigma)$ in $K_2$ whose
sign agrees with that of $\sigma$ in $\Sp^n_{n+2}$. We remove the interiors of $\sigma_1$ and
$\sigma_2$, and then identify their boundaries. This produces a new
triangulation of the $n$-sphere, denoted by $K = K_1 \# K_2$.

We now define a simplicial map
$f : K \longrightarrow \mathbb{S}^n_{n+2}$
by setting $f(v) = f_1(v)$ for vertices $v \in K_1$, and
$f(v) = f_2(v)$ for vertices $v \in K_2$. This gives a well-defined simplicial
map.

To compute the degree of $f$, consider the simplex $\sigma$ in $\Sp^n_{n+2}$ and assume that
$\sigma$ has a positive sign. Since the positive simplices mapping to $\sigma$
in $K_1$ and $K_2$ remain positive in $K$, except for $\sigma_2$ from $K_2$, and the
negative simplices mapping to $\sigma$ in $K_1$ and $K_2$ remain negative in $K$,
except for $\sigma_1$ from $K_2$, it follows that $\deg(f)=d_1+d_2$.
\end{proof}

The next lemma gives a converse analogue of the previous lemma.

\begin{lemma}
Let $n,d\geq 1$. Suppose we have a simplicial map $f: K \to \mathbb{S}^n_{n+2}$ of degree $d$, where $K$ is a triangulation of an $n$-sphere that contains a missing facet. Then $K$ can be expressed as a connected sum of two triangulated $n$-spheres $K_1$ and $K_2$. Moreover, $K_1$ and $K_2$ induce simplicial maps with degrees $d_1$ and $d_2$, respectively, with $d = d_1 + d_2$, where $d_1, d_2 \geq 0$.
\end{lemma}
 \begin{proof}
 We have a simplicial map $f:K\to \mathbb{S}^{n}_{n+2}$, where $K$ is a triangulation of an $n$-sphere containing a missing facet. Then there exist triangulations $K_1$ and $K_2$ of $n$-spheres such that the missing facet appears as an $n$-simplex in each of them and $K = K_1 \# K_2$. Define the simplicial maps $f_1:K_1\to \mathbb{S}^{n}_{n+2}$ by $f_1(v)=f(v)$ for all $v\in K_1$, and $f_2:K_2\to \mathbb{S}^{n}_{n+2}$ by $f_2(u)=f(u)$ for all $u\in K_2$.  Without loss of generality, assume that the vertices of the missing facet are mapped to $v_1, v_2, \dots, v_{n+1}$, which are vertices of $\mathbb{S}^{n}_{n+2} = \partial([v_1 v_2 \dots v_{n+2}])$. Consider the $n$-simplex $\sigma = [v_2 v_3 \dots v_{n+2}]$ in $\mathbb{S}^{n}_{n+2}$. 

We know that $\alpha^{+}(\sigma)-\alpha^{-}(\sigma) = d$ with respect to $f$. Let $\alpha^{+}(\sigma)-\alpha^{-}(\sigma)=d_1$ with respect to $f_1$ and $\alpha^{+}(\sigma)-\alpha^{-}(\sigma)=d_2$ with respect to $f_2$. Then $\deg(f_1)=d_1$ and $\deg(f_2)=d_2$. Therefore, $\deg(f)=d_1+d_2$, and consequently $d=d_1+d_2$.
 \end{proof}

Using Lemma \ref{connectedlemma}, we provide an alternative construction on
the triangulated \(2\)-sphere that is vertex minimal as well as facet minimal and is not combinatorially equivalent to the construction of Madahar and Sarkaria in \cite{Madahar2001}.

\begin{proposition}\label{2sphere}
 For all integers $d\ge 1$, $\Lambda(2,1)=4$, $\Lambda(2,2)=7$, and $\Lambda(2,d)=2d+2$ for $d\ge 3$.
\end{proposition}

\begin{proof}
The identity map $h_1:\Sp^2_{4}\to \Sp^2_{4}$ is the simplicial map of degree $1$.
 Consider the triangulated $2$-sphere $L_2$ on $7$ vertices with facet list $\{[u_{\{1,1\}}u_{\{2,1\}}u_{\{3,1\}}], [u_{\{1,2\}}u_{\{2,2\}}u_{\{3,1\}}], $ $[u_{\{1,1\}}u_{\{2,1\}}u_{\{4,1\}}], [u_{\{1,2\}}u_{\{2,1\}}u_{\{4,1\}}], [u_{\{1,2\}}u_{\{2,2\}}u_{\{4,1\}}], [u_{\{1,2\}}u_{\{2,1\}}u_{\{4,2\}}], [u_{\{1,1\}}u_{\{3,1\}}\\ u_{\{4,1\}}], [u_{\{1,2\}}u_{\{3,1\}}u_{\{4,2\}}], [u_{\{2,2\}}u_{\{3,1\}}u_{\{4,1\}}], [u_{\{2,1\}}u_{\{3,1\}}u_{\{4,2\}}]\}$.  $L_2$ can be constructed by applying three times facet subdivision over the standard $2$-sphere. 
 We define a map $h_2:L_2\to \Sp^2_{4}$ by sending $u_{\{1,1\}},u_{\{1,2\}}$ to $v_1$, $u_{\{2,1\}},u_{\{2,2\}}$ to $v_2$, $u_{\{3,1\}}$ to $v_3$, and $\{u_{\{4,1\}},u_{\{4,2\}}\}$ to $v_4$. We verify that $\deg(h_2)=2$; moreover, the preimage of $[v_1v_2v_4]$ consists of four $2$-simplices, among which $[u_{\{1,2\}}u_{\{2,1\}}u_{\{4,1\}}]$ has the opposite sign as of $[v_1v_2v_4]$ in $\Sp^2_{4}$.

Construct a triangulated $2$-sphere $L_3 = L_2 \# \Sp^2_{4}$ by identifying the face $h_1^{-1}([v_1v_2v_4])$ with a face in $h_2^{-1}([v_1v_2v_4])$ that has the opposite sign in $L_2$ to that of $[v_1v_2v_4]$ in $\Sp^2_{4}$.  From Lemma~\ref{connectedlemma}, there is a simplicial map $h_3:L_3\to \Sp^2_{4}$ of degree $3$. We find that $f_0(L_3)=8$.

Now, inductively construct $2$-spheres $L_{2k+1} = L_{2k-1} \# L_2$ for $k \geq 2$ by identifying 
a face from $h_{2k-1}^{-1}([v_1v_2v_4])$ that has the same sign in $L_{2k-1}$ as that of $[v_1v_2v_4]$ in $\Sp^2_{4}$ with a face from $h_{2}^{-1}([v_1v_2v_4])$ that has the opposite sign in $L_{2}$ as that of $[v_1v_2v_4]$ in $\Sp^2_{4}$. 
From Lemma~\ref{connectedlemma}, there exists a simplicial map $h_{2k+1} : L_{2k+1} \to \Sp^2_{4}$ of degree $2k+1$. 
We observe that $f_0(L_{2k+1}) = 4k + 4$.

Now, we give the construction for the even degrees case. First, we construct a triangulated $2$-sphere $L$ obtained as the suspension of a $4$-cycle with vertex set $\{a_1,b_1,a_2,b_2\}$ and suspension vertices $\{d_1,d_2\}$, and then subdividing the faces $[a_1b_1d_1]$ and $[a_2b_2d_1]$ by the vertices $c_1$ and $c_2$, respectively. 
Thus, $L$ is a $2$-sphere on $8$ vertices. 

Define a map $g:L\to \Sp^2_{4}$ which sends $a_i \mapsto v_1$, $b_i \mapsto v_2$, $c_i \mapsto v_3$, and $d_i \mapsto v_4$ for $1\le i\le 2$. 
We observe that $g$ is a simplicial map and $g^{-1}([v_1v_2v_4])$ contains six $2$-simplices; two with the same sign and four with the opposite sign to that of $[v_1v_2v_4]$ in $\Sp^2_{4}$. Thus, $g$ is a map of degree $2$. Now, we construct a triangulated $2$-sphere $L\#\Sp^2_{4}$ by identifying an oppositely signed face in $g^{-1}([v_1v_2v_4])$ with the face $h_1^{-1}([v_1v_2v_4])$. This $L\#\Sp^2_{4}$ yields a simplicial map of degree $3$. Since $g^{-1}([v_1v_2v_4])$ contains two oppositely signed faces, one oppositely signed face remains in $L\# \Sp^2_{4}$. It can be identified with the face $h_1^{-1}([v_1v_2v_4])$, giving a new triangulated $2$-sphere, say $L_4$, on $10$ vertices, which induces a simplicial map $h_4:L_4\to \Sp^2_{4}$ of degree $4$. 

Now, using the map $h_2$, for $k\geq 3$, we inductively construct triangulated $2$-spheres $L_{2k}=L_{2k-2}\#L_2$ by identifying a face from $h_{2k-2}^{-1}([v_1v_2v_4])$ that has the same sign in $L_{2k-2}$ as that of $[v_1v_2v_4]$ in $\Sp^2_{4}$ with a face from $h_2^{-1}([v_1v_2v_4])$ that has the opposite sign in $L_2$ as that of $[v_1v_2v_4]$ in $\Sp^2_{4}$. 
From Lemma~\ref{connectedlemma}, there exists a simplicial map $h_{2k} : L_{2k} \to \Sp^2_{4}$ of degree $2k$. 
We observe that $f_0(L_{2k}) = 4k + 2$.

The vertex minimality and facet minimality of $L_k$ for $k \geq 3$ follow from the facts that $f_0(L_k)=2k+2$, and that any triangulated $2$-sphere $K$ inducing a simplicial map of degree $k$, say $f:K \to \Sp^2_{4}$, must contain at least $4k$ $2$-simplices and hence at least $2k+2$ vertices. 
\end{proof}

\begin{remark}
   \rm{ The constructions for degree maps up to four in Proposition \ref{2sphere} are the same as the degree maps constructed in \cite{Madahar2001}. However, the constructions of the other degree maps over the $2$-sphere differ from those of Madahar and Sarkaria~\cite{Madahar2001}.

The odd-degree maps constructed in \cite{Madahar2001} are uniquely $4$-colorable, in which one color is assigned to exactly two vertices, while each of the remaining three colors is assigned to more than two vertices. In Proposition \ref{2sphere}, for odd-degree simplicial maps $d = 2k+1$, $k \geq 2$, $k+1$ vertices receive the same color.

In the constructions given in \cite{Madahar2001} for even-degree maps $d=2k$, examining the vertex degrees shows that only two vertices have degree $3k$, while the remaining vertices have degree either $3$ or $6$. When we constructed $L_6$ in Proposition \ref{2sphere}, which induces a simplicial map of degree $6$ over the $2$-sphere, $L_6$ had three vertices of degree $9$. Therefore, the $2$-spheres corresponding to the even-degree maps $d>6$ in Proposition \ref{2sphere} contain at least one vertex whose degree is neither $3k$, $3$, nor $6$. Hence, for degree $d > 4$, our constructions differ from those of Madahar and Sarkaria.}    
\end{remark}

We now extend the constructions to higher dimensions. The following theorem provides a general method to construct simplicial maps from a triangulated $n$-sphere to the standard $n$-sphere with a prescribed degree. In addition, the theorem gives explicit formulas for the number of vertices and $n$-simplices of the domain complex corresponding to the degree of the map.

\begin{theorem}\label{thm:mainconstruction}
Let $l,p,n$ be integers with $l \ge 0$ and $2 \le p \le n$. Then there exists a simplicial map
$g_d \colon K_d \to \mathbb{S}^n_{n+2}$ of degree $d\geq 1$ such that
\begin{enumerate}[(i)]
    \item $f_0(K_d) = (l+1)(n+2)$, $f_n(K_d)=d(n+2)$ if $d = ln + 1$;
    \item $f_0(K_d) = (l+1)(n + 2)+p+1$, $f_n(K_d)=d(n+2)+2(n-(p-1))$ if $d = ln + p$.
\end{enumerate}
\end{theorem}
\begin{proof}
Let $\Sp^n_{n+2}$ be the $n$-sphere for $n \geq 3$ with vertex set 
$V(\Sp^n_{n+2}) = \{v_1, v_2, \ldots, v_{n+2}\}$.

When $l=0$ and $d=1$, we consider the simplicial map $g_1 : K_1 \to \Sp^n_{n+2}$ to be the identity map, where 
$K_1 := \partial(w_1 w_2 \cdots w_{n+2})$. Clearly, $f_0(K_1)=n+2$ and 
$f_n(K_1)=n+2$. 

To construct a simplicial map of degree $2$, we first construct a degree $1$ simplicial map $g : K \to \Sp^n_{n+2}$ with $f_0(K)=n+4$, and then perform a connected sum operation between $K$ and $K_1$. To construct $K$, we begin with the boundary of the $(n+1)$-simplex 
$K=\partial [u_1u_2\cdots u_{n+2}]$. Perform a facet subdivision of the $n$-simplex $[u_2u_3\cdots u_{n+2}]$ by inserting a new vertex $u_1'$. This replaces the simplex with $n+1$ new $n$-simplices. Next, subdivide the $n$-simplex 
$[u_1'u_3u_4\cdots u_{n+2}]$ by inserting another new vertex $u_2'$, which again produces $n+1$ new $n$-simplices. 

Thus, the resulting complex $K$ is a triangulation of the $n$-sphere with $n+4$ vertices and $3n+2$ $n$-simplices. Define a simplicial map $g : K \to \Sp^n_{n+2}$ by sending $u_i \mapsto v_i$ for $1 \le i \le n+2$, and $u_1' \mapsto v_1$, $u_2' \mapsto v_2$. Then $g$ is a simplicial map. Under this map, every $n$-simplex of $\Sp^n_{n+2}$ has exactly three facets in its preimage in $K$, except for the two simplices $[v_1v_3v_4\cdots v_{n+2}]$ and $[v_2v_3\cdots v_{n+2}]$, each of which has exactly one facet in its preimage. Hence, $\deg(g)=1$.

Consider the connected sum $K_2 := K_1 \# K$, where we identify a face from 
$g_1^{-1}([v_1v_2\cdots$
$ v_{n+1}])$ having the same sign as $[v_1v_2\cdots v_{n+1}]$ in $\Sp^n_{n+2}$ and a face from 
$g^{-1}([v_1v_2\cdots v_{n+1}])$ having the opposite sign. Thus, by Lemma~\ref{connectedlemma}, there exists a simplicial map $g_2:K_2\to \Sp^n_{n+2}$ of degree $2$. We obtain $f_0(K_2)=n+5$ and $f_n(K_2)=4n+2$.

Now, under $g_2$, the preimage of exactly $n-1$, $n$-simplices of $\Sp^n_{n+2}$ consists of facets that are oppositely signed in $K_2$. Thus, for $3 \le d \le n+1$, a simplicial map of degree $d$ is obtained inductively by defining 
$K_d := K_{d-1} \# K_1$. At each inductive step, there exists a face $\sigma$ such that $g_{d-1}^{-1}(\sigma)$ contains a facet whose sign is opposite to that of $\sigma$ in $\Sp^n_{n+2}$. We identify such a facet from $g_{d-1}^{-1}(\sigma)$ with the face $g_1^{-1}(\sigma)$. By Lemma~\ref{connectedlemma}, this yields a simplicial map $g_d:K_d\to \Sp^n_{n+2}$ of degree $d$. 

Each connected sum increases the number of vertices by $1$ and the number of $n$-simplices by $n$. Consequently, $f_0(K_d)=n+2+(d+1)$ and 
$f_n(K_d)=d(n+2)+2n-2(d-1)$ for $3\leq d\le n+1$.

Next, we observe that there is no facet $\tau$ in $\Sp^n_{n+2}$ such that $g_{n+1}^{-1}(\tau)$ contains a facet in $K_{n+1}$ whose sign is opposite to the sign of $\tau$ in $\Sp^n_{n+2}$. Therefore, we again consider the $n$-sphere $K$, for which there are $n$ facets whose signs change under the map $g$.

Thus, we obtain a simplicial map 
$g_{n+2}:K_{n+2}\to \Sp^n_{n+2}$ of degree $n+2$, where $K_{n+2}$ is obtained by taking the connected sum of $K_{n+1}$ and $K$ by identifying facet $\sigma_1$ from $K_{n+1}$ and $\sigma_2$ from $K$ such that $\sigma_1$ and $\sigma_2$ are oppositely signed and both map to the same facet $\sigma$ in $\Sp^n_{n+2}$. 

Now observe that there are exactly $n-1$ facets in $K_{n+2}$ whose sign changes under the map $g_{n+2}$. Proceeding in the same way as above, we inductively construct $K_d = K_{d-1} \# K_1$ for $n+3 \le d \le 2n+1$ such that there exists a simplicial map $g_d : K_d \to \Sp^n_{n+2}$. In each inductive step, we use one of the $n-1$ facets of $K_{n+2}$ whose sign changes under the map $g_{n+2}$.

We construct simplicial maps 
$g_d:K_d\to \Sp^n_{n+2}$ for $d=ln+2$ with $l\ge3$ using the maps 
$g_{d-1}:K_{d-1}\to \Sp^n_{n+2}$ and $g:K\to \Sp^n_{n+2}$, where $K_d$ is obtained by taking the connected sum of $K_{d-1}$ and $K$ by identifying facets $\sigma_1$ from $K_{d-1}$ and $\sigma_2$ from $K$ such that $\sigma_1$ and $\sigma_2$ are oppositely signed and both map to the same facet $\sigma$ in $\Sp^n_{n+2}$. 

Similarly, we inductively construct simplicial maps 
$g_d:K_d\to \Sp^n_{n+2}$ for $d=ln+p$, where $l\ge3$ and $3\le p\le n+1$, using the maps 
$g_{d-1}:K_{d-1}\to \Sp^n_{n+2}$ and $g_1:K_1\to \Sp^n_{n+2}$. Here $K_d$ is obtained by taking the connected sum of $K_{d-1}$ and $K_1$ by identifying facets $\gamma_1$ from $K_{d-1}$ and $\gamma_2$ from $K_1$ such that $\gamma_1$ and $\gamma_2$ are oppositely signed and both map to the same facet $\sigma$ in $\Sp^n_{n+2}$. 

Since in each step we use either $K_1$ or $K$ for the connected sum, the number of vertices increases by either $1$ or $3$, and the number of facets increases by either $n$ or $3n$. Hence, the result follows.
 \end{proof}

\begin{corollary}\label{them:upperbound}
Let $l,p,n$ be integers with $l \ge 0$ and $2 \le p \le n$, and $d\geq 1$. Then
\begin{enumerate}[(a)]
    \item for $d = ln + 1$, $\Lambda(n,d) \le n + d + 2l + 1$;
    \item for $d = ln + p$,  $\Lambda(n,d) \le n + d +2l+ 3$.
\end{enumerate}
\end{corollary}
\begin{proof}
   The proof follows from Theorem \ref{thm:mainconstruction}. 
\end{proof}

\subsection{Asymptotics of  $\Lambda(n,d)$}\label{subsection 1}

In this section, we derive some results on $\Lambda(n,d)$.

\begin{proposition}
    Fix $n\geq 3$ and $d\geq 1$. Then $\Lambda(n,d)-1 \leq \Lambda(n,d+1)\leq \Lambda(n,d)+3$.
    \end{proposition}

    \begin{proof}
Let $K_{d+1}$ be a triangulated $n$-sphere with $f_0(K_d)=\Lambda(n,d+1)$ which induces a simplicial map of degree $d+1$, say $g_{d+1}:K_{d+1}\to \mathbb{S}^{n}_{n+2}$. Let $K$ be a triangulated $n$-sphere with $n+2$ vertices which induces a simplicial map of degree $-1$, say $g:K\to \mathbb{S}^{n}_{n+2}$ (see Remark~\ref{remark}).

We construct another triangulated $n$-sphere $L=K_{d+1}\#K$ by identifying a facet $\sigma_1$ from $K_{d+1}$ and a facet $\sigma_2$ from $K$ such that the two facets have opposite signs but map to the same facet in $\Sp^n_{n+2}$ under the corresponding maps. Such a pair of facets exists because $g:K\to \Sp^n_{n+2}$ is a map of degree $-1$. Thus, by Lemma~\ref{connectedlemma}, there exists a simplicial map, say 
$f:L\to \Sp^n_{n+2}$, of degree $d$. This implies that 
$\Lambda(n,d+1)\geq \Lambda(n,d)-1$.

To prove the inequality $\Lambda(n,d+1)\leq \Lambda(n,d)+3$, we proceed as follows. 
Consider the triangulated $n$-sphere $K$ with $n+4$ vertices constructed in the proof of Theorem~\ref{thm:mainconstruction}, which induces a simplicial map of degree $1$, say $h:K\to \mathbb{S}^{n}_{n+2}$. 
Let $f':K_d\to \mathbb{S}^{n}_{n+2}$ be a simplicial map of degree $d$ and $f_0(K_d)=\Lambda(n,d)$. 

Since $K$ contains a facet whose sign changes under the map $h$, we perform the connected sum operation between $K$ and $K_d$ using oppositely signed facets from $K$ and $K_d$ that map to the same facet in $\Sp^n_{n+2}$. This yields another triangulation, say $L'$. By Lemma~\ref{connectedlemma}, there exists a simplicial map 
$f'':L'\to \mathbb{S}^{n}_{n+2}$ of degree $d+1$. Since $L'$ has three more vertices than $K_d$, we obtain 
$\Lambda(n,d+1)\leq \Lambda(n,d)+3$.
\end{proof}

\begin{theorem}
    Fix $d\geq 1$. Then $\frac{\Lambda(n,d)}{n}\to 1$ if $n\to \infty$.
\end{theorem}

\begin{proof}
We know that $n+2\leq \Lambda(n,d)$, and from Corollary \ref{them:upperbound}, we know that $\Lambda(n,d)\leq n+d+\frac{2d}{n}+1\leq n+3d+1$. We conclude that $1+\frac{2}{n}\leq \frac{\Lambda(n,d)}{n} \leq 1+\frac{3d+1}{n}$. Hence, from the sandwich theorem when $n\to\infty$ $\frac{\Lambda(n,d)}{n}\to 1$.    \
\end{proof}

\begin{lemma}\label{degree 2}
Let $K$ be a triangulated $n$-sphere. If the degree of a simplicial map $f\colon K\to \Sp^n_{n+2}$ is $2$, then $K$ contains at least $n+5$ vertices.
\end{lemma}

\begin{proof}
Let $\{v_1, v_2, \ldots, v_{n+2}\}$ be the vertex set of $\Sp^n_{n+2}$. Consider a simplicial map $f \colon K \to \Sp^n_{n+2}$ of degree $2$. 
Consider the $n$-simplex $[v_1v_2\cdots v_{n+1}]$, which is positive in $\Sp^n_{n+2}$. Then its preimage under $f$ contains at least two positive $n$-simplices, and these simplices cannot share any $(n-1)$-dimensional face.

Let $[u_1u_2u_3\cdots u_{n+1}]$ be one such positive simplex in $K$. If the two positive simplices do not share any face or share an $i$-dimensional face with $0 \le i \le n-3$, then $K$ must contain at least three additional vertices other than $u_1,u_2,u_3,\dots,u_{n+2}$. Thus, $K$ has at least $n+5$ vertices.

If they share an $(n-2)$-simplex, assume without loss of generality that $[u'_1u'_2u_3\cdots u_{n+1}]$ is the other positive $n$-simplex. Now consider the simplex $[v_2v_3\cdots v_{n+2}]$ in $\Sp^n_{n+2}$, which is positive (resp.\ negative) when $n$ is odd (resp.\ even). The preimage of this simplex must contain at least two positive (resp.\ negative) $n$-simplices, and these simplices cannot share an $(n-1)$-dimensional face. This forces the existence of another vertex besides the previous $n+4$ vertices. Therefore $f_0(K)\ge n+5$.
\end{proof}

\begin{lemma}\label{degree3}
Let $K$ be a triangulated $n$-sphere. If the degree of a simplicial map $f\colon K\to \Sp^n_{n+2}$ is $3$, then $K$ contains at least $n+6$ vertices.
\end{lemma}

\begin{proof}
Let $\{v_1, v_2, \ldots, v_{n+2}\}$ be the vertex set of $\Sp^n_{n+2}$. Consider a simplicial map $f \colon K \to \Sp^n_{n+2}$ of degree $3$. 
Consider the $n$-simplex $[v_1v_2\cdots v_{n+1}]$ as a positive simplex in $\Sp^n_{n+2}$. The preimage of this simplex contains at least three positive $n$-simplices, say $\sigma_1,\sigma_2,\sigma_3$, and no two of them share an $(n-1)$-dimensional face.

Let $\sigma_1=[u_1u_2u_3\cdots u_{n+1}]$ be one such simplex in $K$. If $\sigma_1$ and $\sigma_2$ do not share any face or share an $i$-dimensional face with $0\le i\le n-4$, then $K$ must contain at least four additional vertices other than $u_1,u_2,u_3,\dots,u_{n+2}$.

If $\sigma_1$ and $\sigma_2$ share an $(n-3)$-dimensional face, then the fact that $\sigma_1$ and $\sigma_3$ do not share an $(n-1)$-dimensional face again forces the existence of at least four additional vertices other than $u_1,u_2,u_3,\dots,u_{n+2}$.

If every pair among $\sigma_1,\sigma_2,\sigma_3$ shares an $(n-2)$-dimensional face, then there are already three additional vertices besides $u_1,u_2,u_3,\dots,u_{n+1}$. Now consider the simplex $[v_2v_3\cdots v_{n+2}]$. Its preimage contains three simplices, say $\tau_1,\tau_2,\tau_3$. From this we conclude that $K$ must contain at least $n+6$ vertices.
\end{proof}

\begin{theorem}
Fix $n\ge 3$. Then $\Lambda(n,2)=n+5$ and $\Lambda(n,3)=n+6$.
\end{theorem}

\begin{proof}
From Corollary~\ref{them:upperbound}, we have $\Lambda(n,2)\le n+5$ and $\Lambda(n,3)\le n+6$. 
On the other hand, Lemmas~\ref{degree 2} and~\ref{degree3} give $\Lambda(n,2)\ge n+5$ and $\Lambda(n,3)\ge n+6$. Hence the result follows.
\end{proof}

The following result holds even without imposing the non-degeneracy condition on the simplicial degree maps, and hence answers one of the questions posed in Section~5 of \cite{Ryabichev}. 

\begin{theorem}\label{openquestion}
 Fix $n\geq 3$. Then $\limsup\limits_{d\to\infty}\frac{\Lambda(n,d)^{\lfloor \frac{n+1}{2}\rfloor}}{d}\neq 0$.
\end{theorem}

\begin{proof}
We prove the result using Proposition~\ref{pr:cyclicpolytope}.  
Let $M$ be an $(n+1)$-dimensional cyclic polytope with $k$ vertices.  
Its boundary $\partial M$ is an $n$-dimensional polytopal sphere.  
The number of $n$-dimensional faces of $\partial M$ is given by
\begin{equation}
f_n(M)=
\begin{cases}
\displaystyle
\binom{k-\left\lfloor \frac{n+1}{2} \right\rfloor}{\left\lfloor \frac{n+1}{2} \right\rfloor}
+\binom{k-\left\lfloor \frac{n+1}{2} \right\rfloor-1}{\left\lfloor \frac{n+1}{2} \right\rfloor-1},
& \text{if $n$ is even}, \\[10pt]
\displaystyle
2\binom{k-\left\lfloor \frac{n+1}{2} \right\rfloor-1}{\left\lfloor \frac{n+1}{2} \right\rfloor},
& \text{if $n$ is odd}.
\end{cases}
\end{equation}

Let $K_d$ be a triangulation of the $n$-sphere with the minimum number of vertices
$\Lambda(n,d)$ that admits a simplicial map of degree $d$ from $K_d$ to $\mathbb{S}^n_{n+2}$.
Then $f_n(K_d)=d(n+2)+A$, where $A$ is free from $d$.

Assume that $n$ is even and set $m=\left\lfloor \frac{n+1}{2} \right\rfloor$, then $f_n(M)=\binom{k-m}{m}+\binom{k-m-1}{m-1}$.
Since $
\binom{k-m}{m}
=\frac{(k-m)(k-m-1)\cdots(k-2m+1)}{m!}
\le \frac{k^{m}}{m!}$,
and $\binom{k-m-1}{m-1}
=\frac{(k-m-1)\cdots(k-2m+1)}{(m-1)!}
\le \frac{k^{m-1}}{(m-1)!}$,
it follows that $f_n(M)\le \frac{1}{m!}\,k^{m}+\frac{1}{(m-1)!}\,k^{m-1}$, when $n$ is even.

When $n$ is odd and let $m=\left\lfloor \frac{n+1}{2} \right\rfloor$, then $f_n(M)= 2\binom{k-m-1}{m}$. Since $\binom{k-m-1}{m}
=\frac{(k-m-1)(k-m-2)\cdots(k-2m)}{m!}
\le \frac{k^{m}}{m!}$. This implies
$f_n(M)\le \frac{1}{m!}\,k^{m}$. 

Consequently, there exists a constant $C>0$, depending only on $n$, such that $f_n(M)\le C\,k^{m}$. Therefore, $f_n(K_d)\le C\,k^{m}$. Recalling that $f_n(K_d)=d(n+2)+A$, we obtain $d(n+2)+A \le C\left(\Lambda(n,d)\right)^{m}$.
Equivalently, this shows that $\Lambda(n,d)^m \ge \frac{(n+2)}{C
}\,d + \frac{A}{C}$ for all sufficiently large d.
In particular, this implies that
\[
\limsup_{d\to\infty}
\frac{(\Lambda(n,d))^{\left\lfloor \frac{n+1}{2} \right\rfloor}}{d}
> 0.
\]
\end{proof}

 \subsection{Facet minimal constructions and asymptotics of $\mathcal{F}(n,d)$}
Apart from vertex minimality, facet minimality is another important concept in combinatorial topology.  It also measures the comlexities of the triangulated manifolds. It is clear that $\mathcal{F}(1,d)=3d$, and from Proposition \ref{2sphere}, it follows that $\mathcal{F}(2,d)=4d$ for $d\geq 1$. In this section, we construct some  facet-minimal triangulations of simplicial degree maps and study the behavior of $\mathcal{F}(n,d)$ with respect to $n$ and $d$. 

We observe that for a simplicial map $f:K\to \Sp^n_{n+2}$ of degree $d\in \mathbb{Z}$, the complex $K$ must contain at least $|d|(n+2)$ facets. The following result constructs several triangulated $n$-spheres for $n\ge 3$ with exactly $d(n+2)$ facets for some large $d$, which induce simplicial maps of degree $d\geq 1$.

\begin{theorem}\label{facetminimal}
Let $n\geq 3$. For degree $d_1 = n(n+1-r) + r$, where $1 \le r \le n$, there exists a simplicial map $h_{d_1} : L_{d_1} \to \mathbb{S}^n_{n+2}$, where $L_{d_1}$ is a triangulation of the $n$-sphere consisting of exactly $d_1(n+2)$ facets, and $f_0(L_{d_1}) = 2n + d_1 - r + 4$.
Moreover, for any degree $d \ge d_1$ with $d=qn+ d_1$, where $q\ge 0$, there exists a simplicial map  $h_d : L_d \to \mathbb{S}^n_{n+2}$ of degree $d$,
where $L_d$ is a triangulation of the $n$-sphere with $d(n+2)$ facets.
\end{theorem}
\begin{proof}

We first consider $1\le r<n$. Consider an $(n-r+4)$-cycle whose vertices are labeled cyclically as 
$\{u_{\{3,1\}}, u_{\{1,1\}}, u_{\{2,1\}}, u_{\{1,2\}}, u_{\{2,2\}}, u_{\{1,3\}}, u_{\{2,3\}} \dots\}$. 
We perform a suspension over this cycle using the vertices $u_{\{4,1\}}$ and $u_{\{4,2\}}$, which yields a triangulated $2$-sphere. To construct an $n$-sphere, we then perform $(n-2)$ successive one-vertex suspensions using the vertices 
$u_{\{5,1\}}, u_{\{6,1\}}, \dots, u_{\{n+2,1\}}$ with respect to the vertex $u_{\{3,1\}}$.

Each vertex $u_{\{i,j\}}$ is assigned the color $a_i$ for $1 \le i \le n+2$. This gives a properly colored triangulation $L'$ of an $n$-sphere. We define a simplicial map 
$f':L'\to \Sp^n_{n+2}$ 
by sending every vertex of color $a_i$ to $v_i$.

\medskip

\noindent
\textbf{Case 1: $n-r+4$ is odd.}  
When $n - r + 4$ is an odd number, the simplices of the type
$\{[u_{\{1,1\}}u_{\{2,1\}}u_{\{3,1\}}u_{\{4,1\}}u_{\{5,1\}}\dots u_{\{n+1,1\}}],
[u_{\{1,2\}}u_{\{2,1\}}u_{\{3,1\}}u_{\{4,1\}}\dots u_{\{n+1,1\}}], $\\$
[u_{\{1,2\}}u_{\{2,2\}}u_{\{3,1\}}u_{\{4,1\}}\dots u_{\{n+1,1\}}], \dots,$
$[u_{\{1,\frac{n-r+3}{2}\}}u_{\{2,\frac{n-r+3}{2}\}}u_{\{3,1\}}u_{\{4,1\}}u_{\{5,1\}}\dots u_{\{n+1,1\}}],$\\$[u_{\{1,1\}}u_{\{2,1\}}u_{\{3,1\}}u_{\{4,2\}}u_{\{5,1\}}\dots u_{\{n+1,1\}}],$ $[u_{\{1,2\}}u_{\{2,1\}}u_{\{3,1\}}u_{\{4,2\}}u_{\{5,1\}}\dots u_{\{n+1,1\}}], $\\$ [u_{\{1,2\}}u_{\{2,2\}}u_{\{3,1\}}u_{\{4,2\}}u_{\{5,1\}}\dots u_{\{n+1,1\}}], \dots,$ \\ $[u_{\{1,\frac{n-r+3}{2}\}}u_{\{2,\frac{n-r+3}{2}\}}u_{\{3,1\}}u_{\{4,2\}}u_{\{5,1\}}\dots u_{\{n+1,1\}}]\}$,
all map to $[v_1 v_2 \dots v_{n+1}]$.

We observe that there are $2(n - r + 2)$ facets in $L'$ mapping to $[v_1 v_2 \dots v_{n+1}]$, among which $(n - r + 2)$ have signs opposite to that of their image facets.

Similarly, the preimage of each facet in $\Sp^n_{n+2}$ contains $2(n - r + 2)$ facets, among which $n - r + 2$ have opposite signs, except for the three facets $[v_1 v_2 v_3 v_5 v_6 \dots v_{n+2}]$, $[v_1 v_3 v_4 v_5 \dots v_{n+2}]$, and $[v_2 v_3 \dots v_{n+2}]$.

In these cases,
$f'^{-1}([v_1 v_2 v_3 v_5 v_6 \dots v_{n+2}]) = \emptyset$, $f'^{-1}([v_1 v_3 v_4 v_5 \dots v_{n+2}]) = $\\ $\{[u_{\{1,1\}}u_{\{3,1\}}u_{\{4,1\}}u_{\{5,1\}}\dots u_{\{n+2,1\}}], [u_{\{1,1\}}u_{\{3,1\}}u_{\{4,2\}}u_{\{5,1\}}\dots u_{\{n+2,1\}}]\}$,
and $\\$ $f'^{-1}([v_2 v_3 \dots v_{n+2}]) 
= \{[u_{\{2,\frac{n-r+3}{2}\}}u_{\{3,1\}}u_{\{4,1\}}u_{\{5,1\}}\dots u_{\{n+2,1\}}], $\\$
[u_{\{2,\frac{n-r+3}{2}\}}u_{\{3,1\}}u_{\{4,2\}}u_{\{5,1\}}\dots u_{\{n+2,1\}}]\}$.

\medskip

\noindent
\textbf{Case 2: $n-r+4$ is even.}  
When $n - r + 4$ is an even number, the $n$-simplices of the following type
$\{[u_{\{1,1\}}u_{\{2,1\}}u_{\{3,1\}}u_{\{4,1\}}u_{\{5,1\}}\dots u_{\{n+1,1\}}],
[u_{\{1,2\}}u_{\{2,1\}}u_{\{3,1\}}u_{\{4,1\}}u_{\{5,1\}}\dots u_{\{n+1,1\}}],$ $\\$
$[u_{\{1,2\}}u_{\{2,2\}}u_{\{3,1\}}u_{\{4,1\}}\dots u_{\{n+1,1\}}], \dots,
[u_{\{1,\frac{n-r+2}{2}\}}u_{\{2,\frac{n-r+2}{2}\}}u_{\{3,1\}}u_{\{4,1\}}\dots u_{\{n+1,1\}}],$ $\\$ $[u_{\{1,\frac{n-r+4}{2}\}}u_{\{2,\frac{n-r+2}{2}\}}u_{\{3,1\}}u_{\{4,1\}}u_{\{5,1\}}\dots u_{\{n+1,1\}}],
[u_{\{1,1\}}u_{\{2,1\}}u_{\{3,1\}}u_{\{4,2\}}u_{\{5,1\}}\dots u_{\{n+1,1\}}],$ $\\$
$[u_{\{1,2\}}u_{\{2,1\}}u_{\{3,1\}}u_{\{4,2\}}u_{\{5,1\}}\dots u_{\{n+1,1\}}],
[u_{\{1,2\}}u_{\{2,2\}}u_{\{3,1\}}u_{\{4,2\}}u_{\{5,1\}}\dots u_{\{n+1,1\}}], \dots,
$ $\\$ $[u_{\{1,\frac{n-r+2}{2}\}}u_{\{2,\frac{n-r+2}{2}\}}u_{\{3,1\}}u_{\{4,2\}}\dots u_{\{n+1,1\}}],
[u_{\{1,\frac{n-r+4}{2}\}}u_{\{2,\frac{n-r+2}{2}\}}u_{\{3,1\}}u_{\{4,2\}}\dots u_{\{n+1,1\}}]\}
$ $\\ $ map to $[v_1 v_2 \dots v_{n+1}]$.

We observe that there are $2(n - r + 2)$ facets mapping to $[v_1 v_2 \dots v_{n+1}]$, among which $(n - r + 2)$ have signs opposite to that of their image facets.

Similarly, the preimage of each facet contains $2(n - r + 2)$ facets, except for three facets. In particular,
$f'^{-1}([v_1 v_2 v_3 v_5 v_6 \dots v_{n+2}]) = \emptyset$ and $f'^{-1}([v_2 v_3 \dots v_{n+2}]) = \emptyset$.

Moreover, $f'^{-1}([v_1 v_3 v_4 v_5 \dots v_{n+2}]) =\{[u_{\{1,1\}}u_{\{3,1\}}u_{\{4,1\}}u_{\{5,1\}}\dots u_{\{n+1,1\}}],$ $\\$
$[u_{\{1,\frac{n-r+4}{2}\}}u_{\{3,1\}}u_{\{4,1\}}u_{\{5,1\}}\dots u_{\{n+1,1\}}],$
$[u_{\{1,1\}}u_{\{3,1\}}u_{\{4,2\}}u_{\{5,1\}}\dots u_{\{n+1,1\}}],$ $\\$ $
[u_{\{1,\frac{n-r+4}{2}\}}u_{\{3,1\}}u_{\{4,2\}}u_{\{5,1\}}\dots u_{\{n+1,1\}}]\}.
$

\medskip

\noindent
In both cases, we obtain $\deg(f')=0$. In total, the complex $L'$ contains 
$2(n-1)(n-r+2)+4$
$n$-simplices, among which 
$(n-1)(n-r+2)+2 =: d_1$
have signs opposite to their images.

We now eliminate these simplices via facet subdivisions. Perform facet subdivisions on these $d_1$ simplices, assigning to each new vertex the color missing from the corresponding simplex. This yields a new triangulation $L_{d_1}$ of the $n$-sphere. Define a simplicial map 
$h_{d_1}:L_{d_1}\to \Sp^n_{n+2}$ 
by sending vertices of color $a_i$ to $v_i$. 

By construction, the preimage of every simplex $\sigma\in \Sp^n_{n+2}$ now contains exactly $d_1$ $n$-simplices, all with the same sign as $\sigma$. Hence,
$\deg(h_{d_1}) = d_1 \quad \text{and} \quad f_n(L_{d_1}) = (n+2)d_1$.

\medskip

Let $d \ge n(n+1-r)+r$ with $d = qn + d_1$, where $q>0$. Consider the simplicial map of degree $n$ from Theorem~\ref{thm:mainconstruction},
$g_n: K_n \to \Sp^n_{n+2}$.
The complex $K_n$ has $n(n+2)+2$ $n$-simplices, and there exists exactly one simplex $\sigma\in K_n$ such that the sign of $g_n(\sigma)$ is opposite to the sign of $\sigma$.

Choose a simplex $\tau\in L_{d_1}$ such that $g_n(\sigma)=h_{d_1}(\tau)$ and the signs of $\sigma$ and $\tau$ are opposite. Identifying $\sigma$ and $\tau$, we form the connected sum 
$L_{d_1} \# K_n$.
By Lemma~\ref{connectedlemma}, this yields a simplicial map 
$h': L_{d_1} \# K_n \to \Sp^n_{n+2}$
of degree $d_1+n$ with 
$f_n(L_{d_1}\#K_n) = (d_1+n)(n+2)$.

Iterating this process $q-1$ times, we obtain a triangulated $n$-sphere 
$$L_d = L_{d_1} \# \underbrace{K_n \# K_n \# \cdots \# K_n}_{q\text{ times}}$$
and a simplicial map 
$h_d:L_d\to \Sp^n_{n+2}$
of degree $d=qn+d_1$ with exactly $d(n+2)$ facets.

\medskip

Now, suppose that $r=n$. 
Here $d_1=2n$. We construct a simplicial map of degree $qn$ for $q\ge 2$. Consider a $2q$-cycle with vertices 
$\{u_{\{1,1\}}, u_{\{2,1\}}, u_{\{1,2\}}, u_{\{2,2\}}, \ldots, u_{\{1,q\}}, u_{\{2,q\}}\}$.
We suspend this cycle using $u_{\{3,1\}}$ and $u_{\{4,1\}}$, and then perform $(n-2)$ successive one-vertex suspensions using 
$u_{\{5,1\}}, \ldots, u_{\{n+2,1\}}$ with respect to $u_{\{3,1\}}$. Let the resulting $n$-sphere be $L''$. Assign colors $a_i$ to $u_{\{i,j\}}$ for $1\le i\le n+2$ and define 
$f'':L''\to \Sp^n_{n+2}$ accordingly.

Then each $n$-simplex in $\Sp^n_{n+2}$ has $2q$ preimages, of which $q$ have opposite signs, except for two simplices whose preimages are empty. Hence $\deg(f'')=0$. 

Proceeding as before via facet subdivisions, we obtain a triangulated $n$-sphere $L_{qn}$ and a simplicial map 
$h_{qn}:L_{qn}\to \Sp^n_{n+2}$
of degree $qn$ with exactly $qn(n+2)$ facets.
\end{proof}

\begin{corollary}\label{facetminimal2}
  Let $a = n^2 + 1$ where $n \geq 3$. Then, for each $d > a$, there exists a simplicial map $f: K \to \Sp^n_{n+2}$ of degree $d$ such that $f_n(K) = d(n+2)$, where $K$ is a triangulated $n$-sphere.
\end{corollary}

\begin{proof}
From Theorem~\ref{facetminimal}, we have simplicial maps 
$h_{d_1}:L_{d_1}\to \Sp^n_{n+2}$ 
of degree $d_1$, where $d_1 = n(n+1-r)+r$ for $1\le r\le n$, and $f_n(L_{d_1})=d_1(n+2)$. 
Further, for each $d = qn + d_1$ with $q\ge 1$, there exists a simplicial map of degree $d$, 
$h_d:L_d\to \Sp^n_{n+2}$, satisfying $f_n(L_d)=d(n+2)$.

Let $p\ge a = n^2+1$. Then we can write 
$p = mn + s$ 
for some integers $m\ge n$ and $1\le s\le n$. If $s = n$, then 
$p = mn + n = (m+1)n$,
and hence $p$ is of the required form. Therefore, we are done.

Now suppose $s \ne n$. Consider 
$d' = n(n+1-s) + s$.  Then we can write 
$d' = q''n + s$, where $q'' = n+1-s$. Hence $s = d' - q''n$. Substituting this into the expression for $p$, we obtain
$p = mn + s = mn + (d' - q''n) = (m - q'')n + d'$. Since $m \ge n$ and $q'' = n+1-s \le n$, we have $m - q'' \ge 0$. 

Thus $p$ is of the form $p = (m-q)n + d'$ with $q \ge 0$ and $d'$ as above. Therefore, by the construction, there exists a simplicial map 
$f:K\to \Sp^n_{n+2}$ 
of degree $p$, and facets $p(n+2)$, where $K$ is a triangulated $n$-sphere.
\end{proof}

\begin{theorem}
Fix $n\geq 3$. Then $\frac{\mathcal{F}(n,d)}{d}\to n+2$ as $d\to \infty$. 
\end{theorem}

\begin{proof}
We know that for all $n$ and $d$, $\mathcal{F}(n,d) \ge d(n+2)$. On the other hand, by Corollary~\ref{facetminimal2}, we have $\mathcal{F}(n,d)=d(n+2)$ for all $d > n^2+1$. Therefore, $ \lim_{d\to\infty}\frac{\mathcal{F}(n,d)}{d}=n+2$.
\end{proof}

\begin{theorem}\label{f(n,d)/n}
    Fix $d \ge 1$, then the function $\mathcal{F}(n,d)$ grows at most linearly in $n$, and its normalized growth satisfies
$d\leq \limsup\limits_{n \to \infty} \frac{\mathcal{F}(n,d)}{n}\leq d+2$.
\end{theorem}

\begin{proof}
We know that for all $n,d$, $\mathcal{F}(n,d)\ge d(n+2)$. On the other hand, Theorem~\ref{thm:mainconstruction} gives $\mathcal{F}(n,d)\le d(n+2)+2\bigl(n-(p-1)\bigr)$, where $p$ is an integer satisfying $2\le p\le n$. From these bounds, we have
$d(n+2)\le \mathcal{F}(n,d)\le d(n+2)+2\bigl(n-(p-1)\bigr)$.
We divide throughout by $n$ to obtain $d+\frac{2d}{n}
\;\le\;
\frac{\mathcal{F}(n,d)}{n}
\;\le\;
(d+2)+\frac{2\bigl(d-(p-1)\bigr)}{n}$.
Taking the limit superior as $n\to\infty$ yields
\[
d\le \limsup_{n\to\infty} \frac{\mathcal{F}(n,d)}{n}\le d+2.
\]
\end{proof}

\begin{corollary}\label{facetminimal23}
   $ \limsup\limits_{n \to \infty} \frac{\mathcal{F}(n,d)}{n}= d+2$ for $2\leq d\leq 3$.
\end{corollary}

\begin{proof}
  For $2 \le d \le 3$, we have $\Lambda(n,d)=n+d+3$. Using Theorem~\ref{thm:mainconstruction}, we obtain $\mathcal{F}(n,d) \le (d+2)n + 2$ for $d=2,3$. On the other hand, Proposition~\ref{pr:pseudomanifold} yields $\mathcal{F}(n,d) \ge n(n+d+3) - (n+2)(n-1)$, which simplifies to $\mathcal{F}(n,d) \ge (d+2)n + 2$. Combining the upper and lower bounds, we conclude that $\mathcal{F}(n,d) = (d+2)n + 2$. In particular, $\limsup_{n \to \infty} \frac{\mathcal{F}(n,d)}{n} = d+2$.
\end{proof}

\section{A characterization of degree maps on spheres}\label{section3}

Let $K$ be a triangulated $n$-sphere and let $f:K\to \Sp^n_{n+2}$ be a simplicial map of degree $d$ for some $d\geq 0$. We know that $K$ is $(n+2)$-colorable. In \cite{Planken}, the author proved that if the $1$-skeleton of an $n$-sphere is $(n+2)$-colorable, then there exists a subdivision of the triangulated sphere in which, for every $(n-2)$-simplex, the number of incident $(n-1)$-simplices is divisible by three.

In this section, we first characterize simplicial degree maps for which the signs of the facets are preserved under the map. Using this characterization, we then describe triangulated $n$-spheres whose $1$-skeletons are $(n+2)$-colorable. In particular, we prove that there exists a subdivision of such triangulated spheres in which, for every $(n-2)$-simplex, the number of incident $(n-1)$-simplices is divisible by three, and for every $(n-3)$-simplex, the number of incident $(n-2)$-simplices is even and cannot be equal to $6$.

\begin{lemma}\label{characterization}
Fix $n\ge 2$. Let $f:K\to \Sp^n_{n+2}$ be a degree $d$ simplicial map such that $f_n(K)$ $=d(n+2)$. Let $v\in \Sp^n_{n+2}$ and let $f^{-1}(v)$ denote its preimage set. Then the following hold:

\begin{itemize}

\item[(i)] For any two vertices $x,y\in f^{-1}(v)$, we have $xy\notin K$.

\item[(ii)] The sign of $f(\sigma)$ in $\Sp^n_{n+2}$ is the same as the sign of $\sigma$ in $K$ for all facets $\sigma\in K$.

\item[(ii)] For every $m$-simplex $\sigma\in K$ with $m\le n-2$, the link $\lk(\sigma)$ has exactly $d''(n-m+1)$ faces of dimension $n-m-1$ for some $d''\le d$, and there exists a simplicial map $g$ of degree $d''$ from $\lk(\sigma)$ to $\Sp^{n-m-1}_{n-m+1}$.

\end{itemize}
\end{lemma}

\begin{proof}

Let $v\in \Sp^n_{n+2}$ and let $x,y\in f^{-1}(v)$. If $xy\in K$, then there exists a facet $\sigma$ of $K$ containing $xy$ such that $f(\sigma)$ is an $(n-1)$-dimensional or lower-dimensional face of $\Sp^n_{n+2}$, contradicting the fact that $f$ is non-degenerate. Hence $xy\notin K$, proving $(i)$.

Suppose that the sign of $f(\sigma)$ is opposite to the sign of $\sigma$ for some facet $\sigma\in K$. Then $f^{-1}(f(\sigma))$ contains more than $d$ facets, implying that $f_n(K)>d(n+2)$, a contradiction. This proves $(ii)$.

Let $\sigma\in K$ be an $m$-simplex with $m\le n-2$. Then $\lk(\sigma)$ is a triangulated $(n-m-1)$-sphere. Since $f$ is non-degenerate, $\lk(\sigma)$ must be $(n-m+1)$-colorable. Let $f(\sigma)=\tau$ where $\tau\in \Sp^n_{n+2}$. Therefore, $f$ induces a simplicial map 
$g:\lk(\sigma,K)\to \lk(\tau,\Sp^n_{n+2})$
defined by $g(x)=f(x)$ for all $x\in \lk(\sigma,K)$. The map $f$ also induces orientations on the facets of $\lk(\sigma,K)$ and $\lk(\tau,\Sp^n_{n+2})$. 

We observe that $g$ is non-degenerate. Let $\gamma$ be a facet of $\lk(\tau,\Sp^n_{n+2})$. Suppose $g^{-1}(\gamma)$ contains a facet $\beta$ whose sign is opposite to that of $\gamma$. Then $\beta\star\sigma$ is a facet of $K$ whose sign is opposite to that of $\gamma\star\tau$ in $\Sp^n_{n+2}$. This is a contradiction.
 Therefore, the sign of $g(\beta)$ is the same as the sign of $\beta$ for all facets $\beta\in \lk(\sigma)$. 

Since $\beta$ was chosen arbitrarily, we obtain $f_{n-m-1}(\lk(\sigma))=d''(n-m+1)$ for some $d''\ge 1$, and $g$ is a simplicial map of degree $d''$. This proves $(iii)$.
\end{proof}

\begin{theorem}
Let $n \geq 3$. Let $K$ be a triangulated $n$-sphere such that the one skeleton of $K$ is proper $(n+2)$-colorable. Then there exists a subdivision $K'$ of $K$ such that the links of every $(n-2)$-face  in $K'$ contain $3l$ vertices for some $l \geq 1$, and the links of each $(n-3)$-face  contain $2m+2$ vertices for some $m \geq 1$ with $m \neq 2$. 
\end{theorem}

\begin{proof}
Let $V(\Sp^n_{n+2})=\{v_1,v_2,\ldots,v_{n+2}\}$.
Let $K$ be a triangulated $n$-sphere that is properly $(n+2)$-colorable. 
Let $\{1,2,\ldots,n+2\}$ be the color set, and assign these $n+2$ colors to the vertices of $K$. 
Thus we can define a simplicial map $f:K\to \Sp^n_{n+2}$ by sending every vertex of color $i$ to $v_i$ for $1\le i\le n+2$.

Using the given orientation on $\Sp^n_{n+2}$, we orient the facets of $K$ as follows: a facet of $K$ is declared positive if the ordering of its vertices induced by the colors agrees with the orientation of the corresponding facet in $\Sp^n_{n+2}$, and negative otherwise.

Suppose the degree of the map $f$ is $d$ for some $d\geq 1$ and $f_n(K)=d(n+2)$. Then from Lemma~\ref{characterization}, the link of every $(n-2)$-face is a $3\ell$-cycle for some $\ell\ge 1$. Moreover, from Lemma~\ref{characterization} and Proposition~\ref{2sphere}, the link of every $(n-3)$-face is a $2$-sphere with $2m+2$ vertices for $m\ge 1$.

We claim that $m=2$ is not possible. Suppose $m=2$. Then there exists an $(n-3)$-simplex, say $\beta$, such that $\lk(\beta)$ is a $2$-sphere with $6$ vertices. This induces a simplicial map 
$g:\lk(\beta)\to \Sp^2_{4}$ 
of degree $2$. However, by Proposition~\ref{2sphere}, there does not exist such a map.
Therefore $m=2$ is not possible.

Let $\sigma\in \Sp^n_{n+2}$ and suppose $f^{-1}(\sigma)$ contains $d_1$ facets, say $\sigma_1,\ldots,\sigma_{d_1}$, such that the sign of $f(\sigma_i)$ is opposite to the sign of $\sigma$ in $\Sp^n_{n+2}$ for each $i$. Perform facet subdivisions on $\sigma_1,\ldots,\sigma_{d_1}$ by inserting $d_1$ new vertices in their interiors and assigning to each new vertex the color missing from the vertices of the corresponding facet.

After this step we obtain a new triangulated sphere $K_1$ with $f_0(K)+d_1$ vertices and $f_n(K)+nd_1$ facets. Moreover, we obtain a new simplicial map 
$g_1:K_1\to \Sp^n_{n+2}$ 
such that all facets in $g_1^{-1}(\sigma)$ have the same sign in $K_1$ as $\sigma$ in $\Sp^n_{n+2}$.

Now repeat this process for all facets of $\Sp^n_{n+2}$ until we obtain a triangulated sphere $L$ such that there exists a simplicial map 
$h:L\to \Sp^n_{n+2}$ 
for which the sign of $h(\eta)$ in $\Sp^n_{n+2}$ is the same as the sign of $\eta$ in $L$ for all facets $\eta\in L$. Thus the degree of the map $h$ is $d$ for some $d\ge 1$ and $f_n(L)=d(n+2)$. 

Now the result follows from the first part of the proof.
\end{proof}

\section{An application to the covering type of Moore spaces}\label{section4}

 In this section, we construct triangulations of the Moore spaces $M(\mathbb{Z}_d,n)$ using the simplicial mapping cone of simplicial degree maps between spheres. Using these triangulations, we obtain a bound on the covering type of $M(\mathbb{Z}_d,n)$. In the following result, we show the existence of such triangulations and describe their construction.

\begin{theorem}\label{pr:triangulation of moore space}
Let $K$ be a triangulation of $n$-sphere, and let  $g:K\to \mathbb{S}^n_{n+2}$ be a simplicial map of degree $d$. Then there exists a triangulation of a Moore space of type $M(\mathbb{Z}_d,n)$ with $f_0(K)+n+3$ vertices.
   \end{theorem}

\begin{proof}

Let $K_d$ be a triangulation of the $n$-sphere such that 
$g \colon K_d \rightarrow \mathbb{S}^n_{n+2}$ is a simplicial map of degree $d$. We first construct the simplicial mapping cone of the map $g$. Consider the staircase triangulation of $K_d \times I$, where 
$I := \bigl\{ \{0\}, \{1\}, \{[01]\} \bigr\}$ is a simplicial triangulation of the unit interval. Next, we take the cone over the boundary component $K_d \times \{0\}$ with a newly introduced vertex $v$. 
On the other end, for each $n$-simplex $\sigma \subset K_d$, we identify $\sigma \times \{1\}$ with its image $g(\sigma)$ in $\mathbb{S}^n_{n+2}$.
This procedure yields a simplicial complex $X$ triangulating the mapping cone. In particular, the number of vertices satisfies
$f_0(X) = f_0(K_d) + n + 3.$
From Lemma~4.2 of \cite{Tinarrage}, the simplicial mapping cone is homotopy equivalent to the mapping cone. Consequently, the simplicial complex $X$ is a triangulation of the Moore space $M(\mathbb{Z}_d,n)$.
\end{proof}

\begin{example}
  \rm{
We begin with a simplicial map $g \colon K \longrightarrow \mathbb{S}^2_4$
of degree \(3\) where $f_0(K)=8$.  Such a map exists from Proposition \ref{2sphere}. Let $V(K)=\{1,2,\dots,8\}$ and \(K\) has the following collection of facets: $\{[1,2,3],[5,6,3],[1,6,7],[1,2,8],[5,6,8],$ $[1,6,4],[1,3,4],[1,7,8],
[5,3,8],$ $[2,3,8],[6,7,8],[6,3,4]\}$. Let the simplicial map \(g\) is defined on vertices by
$\{1,5\}\mapsto 9,\quad
\{2,6\}\mapsto 10,\quad
\{3,7\}\mapsto 11,\quad
\{4,8\}\mapsto 12$.

Next, consider the product \(K \times I\), where
$I := \bigl\{\{a\},\{b\},[ab]\bigr\}$. We equip \(K \times I\) with the staircase triangulation. On the subcomplex \(K \times \{b\}\), each simplex \(\sigma \times \{b\}\) is identified with its image
\(g(\sigma)\). Let the vertex set of \(\mathbb{S}^2_4\) be \(\{9,10,11,12\}\).

On the other end, we cone off \(K \times \{a\}\) with a new vertex \(13\).
Thus, \(\mathrm{cone}(K \times \{a\})\) has the following facets: $\{[1,2,3,13],[5,6,3,13],[1,6,7,13],[1,2,8,13],$ $[5,6,8,13],
[1,6,4,13],$ $[1,3,4,13],[1,7,8,13]$, $[5,3,8,13],
[2,3,8,13],[6,7,8,13],[6,3,4,13]\}$. Finally, after attaching the simplices coming from \(K \times \{b\}\) via the map \(g\), we obtain a simplicial complex whose facets are
\[
\begin{aligned}
\{&
[1,2,3,13],[5,6,3,13],[1,6,7,13],[1,2,8,13],[5,6,8,13],
[1,6,4,13],[1,3,4,13],[1,7,8,13],\\
&[5,3,8,13],[2,3,8,13],[6,7,8,13],[6,3,4,13],
[1,9,10,11],[1,2,10,11],[1,2,3,11],[5,9,10,\\ 11],
&[5,6,10,11],
[5,6,3,11],[1,6,10,11],[1,6,7,11],[1,9,10,12],[1,2,10,12]
\}.
\end{aligned}
\]
This simplicial complex realizes the Moore space \(M(\mathbb{Z}_3,2)\).}

\end{example}

\begin{proposition}
  Let $X$ be a Moore space of type $M(\mathbb{Z}_d,1)$, and $M(\mathbb{Z}_d,2)$. Then $ct(M(\mathbb{Z}_d,1))$ $ \le 3+\left\lceil \frac{3d}{2}\right\rceil$, and  $ct(M(\mathbb{Z}_d,2))\le 4+\left\lceil \frac{3d}{2}\right\rceil$. 
\end{proposition}
\begin{proof}
Consider a $3d$-gon in which the vertices on each side are labeled as  $\{ u_{\{1,1\}}, u_{\{2,1\}},u_{\{3,1\}}, \\ u_{\{1,2\}}, u_{\{2,2\}}, u_{\{3,2\}}, \dots, u_{\{1,d\}}, u_{\{2,d\}},u_{\{3,d\}} \}$.  
Now, we introduce a $\left\lceil \frac{3d}{2} \right\rceil$-cycle inside the $3d$-gon. Label its vertices by $u_4,u_5,u_6,\dots,u_{\left(\left\lceil \frac{3d}{2} \right\rceil + 3\right)}$ in anticlockwise order. We construct a triangulation as follows. 

First, we draw two edges from two adjacent vertices of the $3d$-gon to $u_4$. Without loss of generality, assume these vertices are $u_{\{1,1\}}$ and $u_{\{2,1\}}$. We draw two edges $[u_{\{1,1\}}u_4]$ and $[u_{\{2,1\}}u_4]$. Then draw edges from the next two adjacent vertices, $u_{\{3,1\}}$ and $u_{\{1,2\}}$ to the vertex $u_5$. The edges are $[u_{\{3,1\}}u_5]$ and $[u_{\{1,2\}}u_5]$. After adding edges, we obtain two triangles $[u_{\{1,1\}}u_{\{2,1\}}u_4]$, and $[u_{\{1,2\}}u_{\{3,1\}}u_4]$, and a quadrilateral $[u_{\{1,2\}}u_{\{1,2\}}u_4u_5]$. Then again we add edges between the next two adjacent vetrtices to the vertex $u_6$. We follow this process until we reach the vertex $u_{\left(\left\lceil \frac{3d}{2} \right\rceil + 3\right)}$.

If $d$ is even, then two vertices $u_{\{2,d\}}$, and $u_{\{3,d\}}$ of the $3d$-gon remain that are not adjacent to the vertices $u_4,u_5,\dots u_{\left(\left\lceil \frac{3d}{2} \right\rceil+3\right)}$. We draw edges from $u_{\{2,d\}}$, and $u_{\{3,d\}}$ to $u_{\left(\left\lceil \frac{3d}{2} \right\rceil+3\right)}$.

If $d$ is odd, then one vertex $u_{\{3,d\}}$ of the $3d$-gon remains that is not adjacent to the vertices $\{u_4,u_5,\dots u_{\left(\left\lceil \frac{3d}{2} \right\rceil+3\right)}\}$. We draw an edge from $u_{\{3,d\}}$ and also from the vertex $u_{\{1,1\}}$  to $u_{\left(\left\lceil \frac{3d}{2} \right\rceil+3\right)}$.

Now we are left with $\left\lfloor \frac{3d}{2} \right\rfloor$ quadrilaterals, and the vertices of each quadrilateral are different. To triangulate the quadrilaterals, we draw a diagonal in each quadrilateral and make sure they follow the same pattern, either from top-left to bottom-right or from top-right to bottom-left. Inside the $\left\lceil \frac{3d}{2} \right\rceil$-gon, we add edges to triangulate the interior without adding new vertices.

Now, we identify all the edges of type $u_{{i,j}}u_{{i+1,j}}$ for $i = 1,2$, and $u_{{3,j}}u_{{1,j+1}}$, where $1 \le j \le d$, with indices taken modulo $d$.
 After performing these identifications, the resulting topological space $X$ has a fundamental group $\mathbb{Z}_d$, i.e., $\widetilde{H}_1(X;\mathbb{Z})\cong \mathbb{Z}_d$ and $\widetilde{H}_i(X;\mathbb{Z})=0$ for $i\neq 1$.
This gives a triangulation $K$ of the Moore space $M(\mathbb{Z}_d,1)$ with 
$f_0(X)=\left(\left\lceil \frac{3d}{2} \right\rceil + 3\right)$. Therefore,
$ct(M(\mathbb{Z}_d,1)) \le 3 + \left\lceil \frac{3d}{2} \right\rceil$.

To construct $M(\mathbb{Z}_d,2)$, we perform a one-vertex suspension on $K$.  
This produces a simplicial complex $Y$ that is homotopy equivalent to $M(\mathbb{Z}_d,2)$. 
Therefore, $ct(M(\mathbb{Z}_d,2)) \le 4 + \left\lceil \frac{3d}{2} \right\rceil$. 
\end{proof}

A similar bound in the above proposition was obtained by Raphaël Tinarrage (personal communication).

\begin{proposition}\label{Pr:covering type of moore space}
 Let $X$ be a Moore space of type $M(\mathbb{Z}_d,n)$, where $n\ge 3$ and $d\geq 1$. Then the covering type of $X$ satisfies the following bounds:
\begin{enumerate}[(a)]
  \item $\mathrm{ct}(X) \le 2n + d + 2l + 4$ if $d = ln + 1$, where $l \ge 0$ is an integer.
  \item $\mathrm{ct}(X) \le 2n + d + 2l + 6$ if $d = ln + p$, where $l \ge 0$ is an integer and $2 \le p \le i$.
\end{enumerate}
\end{proposition}
\begin{proof}
   To construct the Moore space $M(\mathbb{Z}_d,n)$, we use the simplicial mapping cone construction. We begin by considering the simplicial map of degree $d$ given in Theorem~\ref{thm:mainconstruction}, $g_d \colon K_d \rightarrow \mathbb{S}^n_{n+2}$, where $K_d$ is a triangulation of the $n$-sphere. By applying Theorem~\ref{pr:triangulation of moore space}, we obtain a Moore space of type $M(\mathbb{Z}_d,n)$ with $f_0(K_d)+n+3$ vertices. We know that $f_0(K_d)\leq \Lambda(n,d)$. Now the result follows from Proposition~\ref{Pr:covering type of moore space} and Corollary~\ref{them:upperbound}.
\end{proof}


\begin{proposition}\label{pr2:Moore Spaces}
   Let $X$ be a Moore space of type $M(\mathbb{Z}^r\oplus\mathbb{Z}_{d_1} 
\oplus\cdots\oplus \mathbb{Z}_{d_m}
,n)$,  $k_0 = min\{ x \geq 0, \binom {n+x}
{n+1} \ge r\}$. Then the covering type  of $X$ satisfies the following bounds: 
\begin{enumerate}[(a)]
 \item $\mathrm{ct}(X) \le k_0 + n + 1 + mn + 5m + k_1 + k_2 + \cdots + k_m + 2(l_1 + l_2 + \cdots + l_m)$, whenever $d_i = l_i n + p_i$ for $1 \le i \le m$, where $l_i$ and $p_i$ are nonnegative integers satisfying $2 \le p_i \le n$.

  \item $\mathrm{ct}(X) \le k_0+n+1+3m+(1+\frac{2}{n})(k_1+k_2+\cdots k_m)+mn(1-\frac{2}{n})$ if $d_i = l_in + 1$  for $1 \le i \le m$, where $l_i\ge 0$ is an integer.   
  \end{enumerate}
\end{proposition}
\begin{proof}
The Moore space $M(\mathbb{Z}^r \oplus \mathbb{Z}_{d_1} \oplus \cdots \oplus \mathbb{Z}_{d_m}, n)$ admits the decomposition
$M(\mathbb{Z}^r \oplus \mathbb{Z}_{d_1} \oplus \cdots \oplus \mathbb{Z}_{d_m}, n)
\simeq M(\mathbb{Z}^r,n) \vee M(\mathbb{Z}_{d_1},n) \vee \cdots \vee M(\mathbb{Z}_{d_m},n).
$ By \cite[Theorem~4.1 and Proposition~4.4]{Govc}, we have 
$\mathrm{ct}\bigl(M(\mathbb{Z}^r,n)\bigr)= k_0 + n + 1$, where $k_0 = \min \left\{ x \ge 0 \,\middle|\, \binom{n+x}{n+1} \ge r \right\}.$
Moreover, for spaces $X$ and $Y$, $\mathrm{ct}(X \vee Y)\le \mathrm{ct}(X) + \mathrm{ct}(Y)- \min\{\mathrm{hdim}(X), \mathrm{hdim}(Y)\} - 1,$
where $\mathrm{hdim}$ denotes the homotopy dimension of the space. 
Now, using Proposition~\ref{Pr:covering type of moore space} for 
$M(\mathbb{Z}_{d_1},n),\ldots, M(\mathbb{Z}_{d_m},n)$, we obtain the required bounds.
\end{proof}


\section{Future Directions}\label{section5}
Constructing manifold triangulations and determining minimal ones have been central themes in combinatorial topology. In this article, we present several constructions of triangulated $n$-spheres that induce simplicial maps of degree $d$ to the standard $n$-sphere for $d\in \mathbb{Z}$ and $n\geq 1$. Our focus has been on both vertex minimality, i.e., the values of $\Lambda(n,d)$, and facet minimality, i.e., the values of $\mathcal{F}(n,d)$, along with their asymptotic behavior with respect to $n$ and $d$.

We also refine the characterization of triangulated $n$-spheres whose $1$-skeleton is $(n+2)$-colorable. Furthermore, we employ the technique of the simplicial mapping cone to construct Moore spaces of a given type from simplicial degree maps between spheres, thereby improving bounds on the covering type of Moore spaces. The results obtained, together with the methods and observations developed in this work, suggest several promising directions for future research.

In subsection \ref{subsection 1}, in the case where no restriction is imposed on the degeneracy of the degree maps, it follows from \cite{Ryabichev} that 
$\limsup_{d\to\infty}\frac{\lambda(n,d)}{d}=0$. The author uses the join operation, which works particularly well because of its ability to maximize the number of facets without significantly increasing the number of vertices of the triangulated sphere. However, the join operation breaks the non-degeneracy of the degree map.

In the case of non-degenerate maps, it is easy to check from Corollary~\ref{them:upperbound} that 
$\frac{\Lambda(n,d)}{d}\to c$ for some constant $c\le 4$ as $n\to\infty$. Moreover, non-degeneracy induces a proper coloring on the triangulated sphere when the simplicial map has non-zero degree. At present, there appears to be no known operation on triangulated spheres that increases the degree of the map quadratically (or faster) while increasing the number of vertices only linearly. This naturally leads to the following question.

\begin{question}
Is $\limsup_{d\to\infty}\frac{\Lambda(n,d)}{d}\neq 0$?
\end{question}

We observe that for $n\ge 3$, $\Lambda(n,d)\le n+d+3$ for $d\le n+1$, and equality holds when $d=2,3$.  Assuming non-degeneracy on degree maps further suggests the following conjecture. 

\begin{conj}\label{conj:n+d+3}
Fix $n\ge 3$. Then $\Lambda(n,d)=n+d+3$ for $d\le n+1$.
\end{conj}

Let $n \geq 3$. From Theorem~\ref{thm:mainconstruction}, Theorem~\ref{facetminimal}, and Corollary~\ref{facetminimal2}, we obtain simplicial degree maps 
$f: K \to \Sp^n_{n+2}$ 
of degree $d$, where $K$ is a triangulated $n$-sphere satisfying $f_n(K) = d(n+2)$, for the following values of $d$:
\begin{itemize}
    \item $d = ln + 1$, for $l \geq 1$,
    \item $d = n(n+1-r) + r$, for $1 \le r \le n$,
    \item $d > n^2 + 1$.
\end{itemize}

Now, consider a simplicial degree $p$ map $q: L \to \Sp^n_{n+2}$ given by Theorem~\ref{thm:mainconstruction}, where $p \le n$. In this case, there exist simplices in $L$ whose orientation (sign) changes under the map $q$ in $\Sp^n_{n+2}$. We observe that such maps can be used to form connected sums with facet-minimal degree maps. Using this idea, one can construct facet-minimal simplicial degree maps for certain values of $d$ between $n+1$ and $n^2+1$, excluding those of the form $ln+1$ (for $l \geq 2$) and $n(n+1-r)+r$ (for $1 \le r \le n$). This observation leads to the following conjecture.

\begin{conj}
Fix $n \geq 3$. For every integer $d$ with $n+2 \le d \le n^2$, there exists a simplicial degree map 
$f: K \to \Sp^n_{n+2}$ 
of degree $d$, where $K$ is a triangulated $n$-sphere satisfying $f_n(K) = d(n+2)$.
\end{conj}

It is also natural to ask whether facet-minimal degree maps are realizable for degrees $d \le n$.

In Theorem~\ref{f(n,d)/n}, we proved that
$d \leq \limsup\limits_{n \to \infty} \frac{\mathcal{F}(n,d)}{n} \leq d+2$. In Corollary~\ref{facetminimal23}, we showed that $\limsup\limits_{n \to \infty} \frac{\mathcal{F}(n,d)}{n} = d+2$ for $d=2,3$. This motivates the following conjecture.

\begin{conj} Let $d\geq 4$. Then
$\limsup\limits_{n \to \infty} \frac{\mathcal{F}(n,d)}{n} = d+2$.
\end{conj}

Assume that Conjecture~\ref{conj:n+d+3} holds. Then $\Lambda(n,d) = n+d+3$ for $d \leq n$. Hence, by Proposition~\ref{pr:pseudomanifold}, we obtain $\mathcal{F}(n,d) \geq n(n+d+3) - (n+2)(n-1) = n(d+2) + 2$ for $d \leq n$.
Combining this lower bound with the inequality $d \leq \limsup\limits_{n \to \infty} \frac{\mathcal{F}(n,d)}{n} \leq d+2$, we conclude that $\limsup\limits_{n \to \infty} \frac{\mathcal{F}(n,d)}{n} = d+2$.

Apart from studying the face numbers and their asymptotic behavior for simplicial degree maps, this article raises the following important observation. In \cite{Madahar2001}, the author constructed simplicial degree maps for triangulated $2$-spheres, where most of the constructions use stacked spheres. In Theorem~\ref{2sphere}, we construct non-stacked spheres that induce simplicial degree maps for all even degrees $\geq 2$. In higher dimensions, the constructions given in Theorem~\ref{thm:mainconstruction} yield only stacked spheres, whereas those in Theorem~\ref{facetminimal} produce non-stacked spheres. This suggests that, for a given simplicial degree $d$ map $f:K \to \Sp^n_{n+2}$, the triangulated $n$-sphere $K$ may be either stacked or non-stacked.
Another observation is that all constructions of simplicial degree maps in this article, as well as in \cite{Apolonskaya, Madahar2001, Musin, Ryabichev}, yield polytopal spheres. This leads to the following natural question.

\begin{question}
Fix $n \geq 3$ and $d \geq 1$. Let $K$ be a triangulated $n$-sphere, and let $f:K \to \Sp^n_{n+2}$ be a simplicial map of degree $d$. Must $K$ be a polytopal sphere?
\end{question}

\smallskip

 \noindent {\bf Acknowledgement:} We would like to thank Raphaël Tinarrage for discussing a possible application of degree maps to the covering type of Moore spaces and for suggesting possible improvements to the bounds given in \cite{Govc}. The second author is supported by the 
  Ministry of Education of the Slovak Republic under grant VEGA 2/0056/25.
   The institute fellowship at IIT Delhi, India, supports the third author.

\end{document}